\documentclass[final,hidelinks,onefignum,onetabnum]{siamart220329}

\usepackage{amsfonts}
\usepackage{graphicx}
\usepackage{epstopdf}
\ifpdf
\DeclareGraphicsExtensions{.eps,.pdf,.png,.jpg}
\else
\DeclareGraphicsExtensions{.eps}
\fi

\usepackage{booktabs}
\usepackage{hyperref}
\usepackage{stmaryrd}
\usepackage{bm}
\usepackage{multirow}
\usepackage{enumitem}
\usepackage{amsopn}
\usepackage{threeparttable}
\usepackage{amsmath}
\usepackage{algorithm,algorithmicx,algpseudocode}

\newcommand{\bfA}{\mathbf{A}}
\newcommand{\bfB}{\mathbf{B}}
\newcommand{\bfC}{\mathbf{C}}

\newsiamthm{test}{Test}
\newsiamthm{assumption}{Assumption}

\def\jump#1{\llbracket #1 \rrbracket }

\newsiamremark{remark}{Remark}
\newsiamremark{hypothesis}{Hypothesis}
\crefname{hypothesis}{Hypothesis}{Hypotheses}
\newsiamthm{claim}{Claim}
\newsiamremark{exmp}{Example}

\headers{Compact FAM with single-step OE}{C.~Fan and K.~Wu}

\title{Compact Finite-Average-Moment Schemes with Single-Step Oscillation Elimination for Hyperbolic Conservation Laws\thanks{This work was partially supported by Science Challenge Project (No.~TZ2025007) and the Shenzhen Science and Technology Program (Grant Nos.~JCYJ20250604144300001 and RCJC20221008092757098).}}

\author{Chuan Fan\thanks{Key Laboratory of Mathematical Modelling and High Performance Computing of Air Vehicles (NUAA), MIIT, Nanjing University of Aeronautics and Astronautics, Nanjing, Jiangsu 210016, China (\email{cfan\_math@nuaa.edu.cn}).}
\and Kailiang Wu\thanks{Corresponding author. Department of Mathematics and Shenzhen International Center for Mathematics, Southern University of Science and Technology, Shenzhen, Guangdong 518055, China (\email{wukl@sustech.edu.cn}).}}

\ifpdf
\hypersetup{
	pdftitle={A Compact Finite-Average-Moment Method with Single-Step Oscillation Elimination for Hyperbolic Conservation Laws},
	pdfauthor={Chuan Fan and Kailiang Wu}
}
\fi

\newsiamthm{property}{Property}
\usepackage{amssymb}
\allowdisplaybreaks

\usepackage{caption}

\begin{document}

\maketitle

\begin{abstract}
High-order shock-capturing schemes often face severe computational bottlenecks due to wide stencils, stagewise nonlinear weights, and expensive local characteristic decompositions. To address this, we propose a compact finite-average-moment (FAM) framework for hyperbolic conservation laws that completely decouples the formal spatial order from the number of evolved local degrees of freedom. By evolving only a $\mathbb{P}^1$ moment state (i.e., cell averages and scaled first-order moments), we achieve up to sixth-order accuracy ($k=3, 4, 5, 6$) via compact, polynomially exact \emph{linear} moment reconstructions. A key algorithmic innovation is consolidating the nonlinear stabilization into a single, derivative-free oscillation-elimination (OE) procedure applied only after the final Runge--Kutta stage. This OE procedure exactly preserves cell averages while applying an explicit exponential correction to the first-order moments, avoiding repeated stagewise nonlinear weights or limiting. Fourier analysis of the linear backbone demonstrates $k$th-order accuracy, strong spurious mode damping, and $(k+1)$th-order cell-average superconvergence. Extensive numerical experiments on scalar laws and the Euler equations verify the expected smooth accuracy and robust, nonoscillatory shock resolution. Notably, by performing componentwise reconstruction directly in conservative variables and streamlining stabilization, the FAM schemes deliver significantly lower complete-run wall-clock times compared to the multi-resolution weighted essentially non-oscillatory (MR-WENO), unified-stencil Hermite WENO (HWENO-U), and OE-HWENO methods.
\end{abstract}



\section{Introduction}

Hyperbolic conservation laws model transport and wave propagation in a wide range of applications, including compressible flows, astrophysics, traffic dynamics, and geophysical transport. Even from smooth data, solutions may develop shocks, contacts, and rarefaction fans in finite time. Their numerical approximation therefore requires methods that are simultaneously high-order accurate in smooth regions, nonoscillatory near discontinuities, and computationally efficient on large meshes.

We consider the $d$-dimensional system
\begin{equation}\label{eq:HCLs}
\partial_t{\bm u}+\nabla\cdot\bm f({\bm u})=0,
\qquad
(\bm x,t)\in\mathbb R^d\times\mathbb R^+,
\end{equation}
where ${\bm u}\in\mathbb R^N$ and $\bm f({\bm u})\in\mathbb R^{N\times d}$. A large body of high-order methods has been developed for \eqref{eq:HCLs}, including finite difference methods \cite{Harten,JS,Pirozzoli,WT}, finite volume methods \cite{HS,LOC,Shu1998}, and discontinuous Galerkin (DG) methods \cite{CS2,DBTM,LLS2,LLS1,PSW}. Essentially non-oscillatory (ENO) and weighted ENO (WENO) ideas \cite{Harten,harten1987,shu1988,JS,Shu1998} remain especially influential because they retain high-order accuracy away from discontinuities while reducing oscillations near nonsmooth features. This accuracy can, however, come with wider stencils, more elaborate boundary closures, and more expensive reconstructions.

Hermite WENO (HWENO) schemes \cite{QS1} improve compactness by evolving both solution averages and derivative information, but they also require nonlinear reconstructions based on smoothness indicators and nonlinear weights. For systems, many WENO/HWENO implementations use local characteristic decomposition to control oscillations \cite{QS2002}, which can be costly in multiple dimensions. Hybrid WENO/HWENO variants \cite{LQ_hybrid,Pirozzoli,ZCQ} reduce this expense by switching to linear reconstructions in smooth regions, although their nonsmooth-region treatment remains nonlinear and, when characteristic variables are used, still incurs local characteristic-variable work.

A complementary line of work uses oscillation-elimination (OE) corrections. Inspired by the oscillation-free DG technique \cite{LLS1,LLS2}, Peng, Sun, and Wu \cite{PSW} introduced the oscillation-eliminating DG (OEDG) framework, in which a DG evolution is followed by a non-intrusive damping step that is scale- and evolution-invariant and admits a closed-form damping update. OE ideas have since been extended to OE-HWENO schemes \cite{FanWu}, bound-preserving unstructured DG methods \cite{DCW}, and applications including relativistic hydrodynamics \cite{CPW}, ideal magnetohydrodynamics \cite{LW_MHD}, two-phase flows \cite{YAW}, and traffic models \cite{CCWX}. DG-based correction frameworks such as COS(DG) \cite{CHLW} further illustrate the effectiveness of OE in a simple convex combination form. These DG-based corrections do not intrinsically require local characteristic decomposition; their main distinction from the present FAM design is that they evolve a higher-order DG polynomial state with more degrees of freedom (DOFs) per cell. In the OE-HWENO Euler implementation compared below, by contrast, the HWENO reconstruction and the local characteristic-variable treatment are retained during the RK stages.

We develop a class of \emph{finite-average-moment} (FAM) schemes aimed at a different operating point. The evolved state is fixed at the $\mathbb P^1$ moment level, compact \emph{linear} reconstruction provides orders $k=3,4,5,6$, and the OE component is applied only once, as a derivative-free damping step after the final SSP-RK3 stage. Thus, FAM evolves only cell averages and scaled first-order moments---two DOFs per conserved component per cell in one dimension and three in two dimensions. By contrast, a total-degree DG state of degree $p$ carries $p+1$ and $(p+1)(p+2)/2$ DOFs per component, respectively, while a tensor-product $\mathbb Q^p$ state carries $(p+1)^2$ DOFs in two dimensions. The target order is therefore decoupled from the dimension of the evolved local state, reducing RK-stage reconstruction work and, when characteristic variables are used, the associated state-dependent transformations.

Table~\ref{tab:intro-positioning} summarizes this algorithmic positioning. FAM combines a fixed $\mathbb P^1$ evolved state, compact linear RK-stage reconstruction, and a single end-of-step OE correction. In the Euler tests reported here, reconstruction is performed componentwise in conservative variables, without local characteristic decomposition. This distinguishes FAM from the compared HWENO-type baselines, which use nonlinear HWENO reconstruction, and from DG-based OEDG/COS(DG) frameworks, whose OE/COS postprocessing is not intrinsically tied to local characteristic decomposition but whose evolved state is a richer DG polynomial state. Related efforts to reduce local characteristic-variable work include \cite{XS,LZ,WS}.

\begin{table}[t]
\centering
\caption{Algorithmic positioning of the proposed FAM framework relative to several closely related compact high-order approaches discussed in the paper.}
\label{tab:intro-positioning}
\footnotesize
\setlength{\tabcolsep}{2pt}
\renewcommand{\arraystretch}{1.66}
\begin{tabular}{@{}p{2.6cm}p{2.5cm}p{2.75cm}p{2.0cm}p{2.6cm}@{}}
\toprule[1.2pt]
Method & Evolved local state & RK-stage treatment & Characteristic decomposition & OE frequency \\
\midrule
\shortstack[l]{FAM (this work)} & fixed $\mathbb P^1$ moments & compact linear & not required & \shortstack[l]{once per time step} \\
\shortstack[l]{HWENO-U\cite{FQZ}} & fixed $\mathbb P^1$ moments & nonlinear HWENO & required & none \\
\shortstack[l]{OE-HWENO\cite{FanWu}} & fixed $\mathbb P^1$ moments & nonlinear HWENO & required & \shortstack[l]{Stagewise} \\[2mm]
\shortstack[l]{OEDG \cite{PSW}\\ COS(DG) \cite{CHLW}} & \shortstack[l]{DG polynomial\\ coefficients} & \shortstack[l]{DG polynomial\\ state evolution} & not required & \shortstack[l]{Stage-wise or\\ end-of-step} \\
\bottomrule[1.2pt]
\end{tabular}
\end{table}

The main contributions are as follows.
\begin{itemize}[leftmargin=20pt]
	\item[(1)] \textbf{A compact fixed-moment FAM framework.}
	We develop FAM schemes of orders $k=3,4,5,6$ while evolving only cell averages and scaled first-order moments. The number of evolved DOFs per cell is fixed, and the target order is supplied entirely by compact linear moment reconstruction.
	
	\item[(2)] \textbf{A single-step derivative-free OE correction.}
	The OE correction is applied only once, after the final RK stage. Its damping coefficients are computed from reconstructed jumps, and the resulting explicit exponential update preserves cell averages exactly, damps scaled first-order moments nonexpansively, and remains high-order small in smooth regions.
	
	\item[(3)] \textbf{Polynomial exactness and linear-backbone analysis.}
	We prove unique solvability and polynomial exactness of the compact reconstruction systems on the chosen stencils. For the linear-advection backbone without OE, Fourier analysis identifies a $k$th-order physical mode, a strongly damped spurious mode, $(k+1)$th-order cell-average superconvergence, and sampled admissible CFL estimates.
	
	\item[(4)] \textbf{Benchmark validation and runtime attribution.}
	Scalar and Euler tests document the designed smooth accuracy, nonoscillatory shock resolution, and the practical role of the end-of-step OE correction. Matched-order comparisons with MR-WENO3/MR-WENO5 and same-$\mathbb P^1$-state comparisons with HWENO-U/OE-HWENO are used to assess the effect of replacing stagewise nonlinear HWENO reconstruction (and, for OE-HWENO, stagewise OE) with compact linear RK-stage reconstruction and a single end-of-step OE correction. Additional timing studies attribute the measured costs to reconstruction, OE, and other update components, and quantify the cost of switching FAM5 from conservative to local characteristic variables.
\end{itemize}
 
The rest of the paper is organized as follows. Section~\ref{sec:FAM} presents the FAM methodology and reconstruction-consistency results. Section~\ref{sec:FAM_SaE} gives the Fourier-based accuracy and CFL analysis for the linear-advection backbone without OE\@. Section~\ref{sec:NumTest} reports the numerical results, and Section~\ref{sec:conclusion} concludes the paper.

\section{\texorpdfstring{FAM formulation, OE properties, and reconstruction consistency}{FAM formulation, OE properties, and reconstruction consistency}}\label{sec:FAM}

This section gives the FAM formulation, the single end-of-step OE correction, the componentwise system implementation, and the reconstruction exactness results supporting the fixed-moment design for orders $k=3,4,5,6$.

Throughout, the evolved unknowns are the cell averages and scaled first-order moments, i.e., the moment variables associated with a $\mathbb P^1$ DG representation. Higher-order accuracy is recovered by compact linear moment reconstructions with the required target-degree exactness; in one dimension the returned reconstructions are degree-$(k-1)$ polynomials, whereas the Cartesian two-dimensional construction uses selected incomplete higher-degree ansatzes to close the square moment systems while retaining exactness through total degree $k-1$. The two-dimensional formulation is restricted to uniform Cartesian grids. The reconstructed polynomials are used for numerical fluxes and volume terms.

Across FAM3--FAM6, the evolved moments, the SSP-RK3 time integrator, the numerical flux, and the OE correction structure all remain unchanged; only the compact linear reconstruction operator and the coefficient $\gamma_k$ vary with the target order.

\subsection{One-dimensional backbone and single-step OE}\label{sec:FAM_1d}

Consider the one-dimensional scalar conservation law
\begin{equation}\label{eq1:1dHCLS}
\partial_t u+\partial_x f(u)=0,
\qquad
u_0(x)=u(x,0),
\qquad
(x,t)\in\Omega\times[0,T].
\end{equation}
The computational domain $\Omega=[a,b]$ is partitioned into a uniform mesh $\Omega=\bigcup_{i=1}^{N_x}I_i$, where each cell $I_i=[x_{i-\frac12},x_{i+\frac12}]$ has width $\Delta x=x_{i+\frac12}-x_{i-\frac12}$ and center $x_i=\tfrac12(x_{i-\frac12}+x_{i+\frac12})$. Multiplying \eqref{eq1:1dHCLS} by a test function $\phi(x)$ and integrating over $I_i$ yields the weak formulation
\begin{equation}\label{eq:weakform_1d}
\int_{I_i}\partial_t u\,\phi(x)\,\mathrm dx
-\int_{I_i}f(u)\frac{\mathrm d\phi(x)}{\mathrm dx}\,\mathrm dx
+f\big(u(x_{i+\frac12},t)\big)\phi(x_{i+\frac12})
-f\big(u(x_{i-\frac12},t)\big)\phi(x_{i-\frac12})=0.
\end{equation}

To attain $k$th-order spatial accuracy, we evolve only the cell average and the scaled first-order moment, and recover order-$k$ accuracy via a compact reconstruction of degree-$(k-1)$ polynomials on each cell. Specifically, let $\{u_i^{(l)}(t)\}_{l=0}^1$ denote approximations of the cell average ($l=0$) and scaled first-order moment ($l=1$) of $u(x,t)$ over $I_i$, defined by
\[
u_i^{(l)}(t)\approx \frac{1}{\Delta x}\int_{I_i}u(x,t)\,\phi_i^{(l)}(x)\,\mathrm dx,
\qquad
\phi_i^{(0)}(x)=1,
\qquad
\phi_i^{(1)}(x)=\frac{x-x_i}{\Delta x}.
\]
Within the $\mathbb P^1$ DG framework, the test function $\phi(x)$ in \eqref{eq:weakform_1d} is taken from $\{\Delta x^{-1},(x-x_i)/\Delta x^2\}$, and $\{u_i^{(l)}(t)\}_{l=0}^1$ are evolved by the semi-discrete scheme
\begin{equation}\label{eq:1dSemi}
\frac{\mathrm d}{\mathrm dt}u_i^{(l)}(t)=\frac{1}{\Delta x}
\left(
\int_{I_i}f(u_h(x,t))\frac{\mathrm d\phi_i^{(l)}}{\mathrm dx}\,\mathrm dx
-\hat f_{i+\frac12}\phi_i^{(l)}(x_{i+\frac12}^-)
+\hat f_{i-\frac12}\phi_i^{(l)}(x_{i-\frac12}^+)
\right),
l=0,1,
\end{equation}
where the integrals are evaluated by suitable quadrature rules, $\hat f_{i+\frac12}=\hat f(u^-_{i+\frac12},u^+_{i+\frac12})$ denotes a monotone and consistent numerical flux, and $u_h(x,t)$ is the reconstructed polynomial used both for the volume term and for the interface traces.

Equation \eqref{eq:1dSemi} has a $\mathbb P^1$ DG-style weak form, but the evolved unknowns are scaled moments rather than orthonormal modal coefficients. They determine the local $\mathbb P^1$ moment representation $u_i^{(0)}+12u_i^{(1)}\phi_i^{(1)}(x)$ on $I_i$, since $\int_{I_i}(\phi_i^{(1)}(x))^2\,\mathrm dx=\Delta x/12$. In what follows, $u_h$ denotes the compact high-order reconstruction used for fluxes, volume terms, and OE jump indicators; the OE correction is applied only after the final RK stage.

Accordingly, the present OE step is invoked only once per time step and is driven by derivative-free jumps of the final reconstructed solution. This differs from OE-HWENO \cite{FanWu}, where OE is applied at RK stages together with HWENO reconstruction in the compared implementation, and from DG-based OEDG/COS(DG) postprocessing in which the damped object is the DG polynomial state rather than the fixed $\mathbb P^1$ moment state used here.

At time $t^n$, let $\bm U_i^n$ denote the moment vector after the final RK stage and before the OE correction, and let $\bm U_i^{n,\sigma}$ denote the OE-corrected moments. Define $\bm U_i(t)=\big(u_i^{(0)}(t),u_i^{(1)}(t)\big)^\top$ and $\bm{\mathcal R}_i(u_h)=\big(\mathcal R_i^{(0)}(u_h),\mathcal R_i^{(1)}(u_h)\big)^\top$. For an explicit $m$-stage Runge--Kutta method, one FAM time step from $t^{n-1}$ to $t^n=t^{n-1}+\Delta t$ reads
\begin{equation}\label{eq3:1dRK}
\bm U_i^{n,r+1}=\sum_{s=0}^{r}\Big(c_{rs}\bm U_i^{n,s}+\Delta t\,d_{rs}\bm{\mathcal R}_i(u_h^{n,s})\Big),
\qquad r=0,\dots,m-1,
\end{equation}
with $\bm U_i^{n,0}=\bm U_i^{n-1,\sigma}$ and $u_h^{n,s}$ denoting the compact degree-$(k-1)$ polynomial reconstruction, viewed as an approximation in a linear polynomial space, associated with the stage moments $\{\bm U_i^{n,s}\}$. The coefficients $c_{rs}$ and $d_{rs}$ are those of the chosen explicit RK scheme written in Shu--Osher form; for the strong-stability-preserving RK methods used in this paper they satisfy $c_{rs}\ge0$, $d_{rs}\ge0$, and $\sum_{s=0}^{r}c_{rs}=1$ for each $r$. After the final stage we set $\bm U_i^n:=\bm U_i^{n,m}$ and define the final solution by the end-of-step OE correction
\begin{equation}\label{eq4:1dOE}
\bm U_i^{n,\sigma}=\mathcal F_{\rm OE}\{\bm U_\kappa^n\}_{\kappa\in\mathcal{S}_i},~
\mathcal{S}_i= \bigcup^1_{m=-1} \Lambda_{i+m}, ~\Lambda_{i}=\{i-1,i,i+1\}.
\end{equation}

The OE operator $\mathcal F_{\rm OE}$ in \eqref{eq4:1dOE} is defined by evaluating, at pseudo-time $\tau=\Delta t$, the solution of the damping equations
\begin{equation}\label{eq7:OE1d_ODE}
\left\{
\begin{aligned}
&\partial_\tau u_\sigma(x,\tau)=-\frac{\beta_i\sigma_i(u_h^*)}{\Delta x}\big(u_\sigma(x,\tau)-P^0u_\sigma(x,\tau)\big), &&(x,\tau)\in I_i\times(0,\Delta t],\\
&u_\sigma(x,0)=u_h^*(x),
\end{aligned}
\right.
\end{equation}
where $\tau\in[0,\Delta t]$ is a pseudo-time variable, $\beta_i=\max_{m\in\Lambda_i}|f'(u_m^{n,(0)})|$ denotes a local wave-speed estimate over $\Lambda_i$, $P^0$ is the standard $L^2$ projection onto piecewise constants, and $u_h^*$ denotes the compact reconstructed solution obtained from the pre-OE moments $\{\bm U_i^n\}$ after the final RK stage. Specifically, $P^0u|_{I_i}=u_i^{(0)}$ represents the cell average of $u$ on $I_i$. The damping coefficient $\sigma_i(u_h^*)$ must be sufficiently small in smooth regions so as to preserve high-order accuracy and large enough near nonsmooth features to suppress spurious oscillations. Although $u_h^{*}$ is the compact high-order reconstruction used to compute the jump indicators, the OE output is defined only through its induced $P^1$ moment state; higher reconstructed modes are not stored after the OE step. Table~\ref{tab:sigma1d} presents the two specific forms used in this paper. 

\begin{table}[t]
\centering
\caption{Two definitions of the damping coefficient $\sigma_i(u_h^*)$ used in \eqref{eq7:OE1d_ODE}.}
\label{tab:sigma1d}
\small
\resizebox{\textwidth}{!}{
\begin{tabular}{c|l|l}
\toprule[1.2pt]
Type & $\sigma_i(u_h^*)$ & Reference quantity \\
\midrule
Type I &
$\displaystyle
\sigma_i(u_h^*)=
\left\{
\begin{aligned}
&0, &&\text{if }\max_{1\le p\le N_x}|u_p^{n,(0)}-\bar u_\Omega|<\epsilon_\Omega,\\
&\gamma_k\,
\frac{|\jump{u_h^*}_{i-\frac12}|+|\jump{u_h^*}_{i+\frac12}|}{\max_{1\le p\le N_x}|u_p^{n,(0)}-\bar u_\Omega|}, &&\text{otherwise},
\end{aligned}
\right.$
&
$\displaystyle \bar u_\Omega=\frac{1}{N_x}\sum_{p=1}^{N_x}u_p^{n,(0)}$ \\
\midrule
Type II &
$\displaystyle
\sigma_i(u_h^*)=
\left\{
\begin{aligned}
&0, &&\text{if }\max_{p\in\Lambda_i}|u_p^{n,(0)}-S_{\Lambda_i}|<\epsilon_i,\\
&\gamma_k\,
\frac{|\jump{u_h^*}_{i-\frac12}|+|\jump{u_h^*}_{i+\frac12}|}{\max_{p\in\Lambda_i}|u_p^{n,(0)}-S_{\Lambda_i}|}, &&\text{otherwise},
\end{aligned}
\right.$
&
$\displaystyle S_{\Lambda_i}=\sum_{p\in\Lambda_i}u_p^{n,(0)}$ \\
\bottomrule[1.2pt]
\end{tabular}}
\end{table}

In Table~\ref{tab:sigma1d}, the zero-denominator tolerance is chosen with the corresponding scale,
\[
\epsilon_\Omega=10^{-14}\max\{\max_p |u_p^{n,(0)}|,1\},
\qquad
\epsilon_i=10^{-14}\max\{\max_{p\in\Lambda_i}|u_p^{n,(0)}|,1\},
\]
for Type~I and Type~II, respectively. Also, $\jump{u_h^*}_{i+\frac12}=u_h^*(x_{i+\frac12}^+)-u_h^*(x_{i+\frac12}^-)$ denotes the interface jump. For Type~II, $S_{\Lambda_i}:=\sum_{p\in\Lambda_i}u_p^{n,(0)}$ is intentionally unnormalized, so the denominator scales with the local background level and weakens damping of small-amplitude structures on a large state. Type~II is therefore the default coefficient in the multiscale tests, whereas Type~I is appropriate when the global variation is the intended scale. Unless stated otherwise, we use Type~II with $\gamma_3=\gamma_4=25$ and $\gamma_5=\gamma_6=12.5$ in one dimension. Sensitivity plots and jump formulas are provided in the supplementary material.

\begin{proposition}
\label{thm:1dOE}
\textbf{Exact update of the OE step.} Let $\bm U_i^n=(u_i^{n,(0)},u_i^{n,(1)})^\top$ and $\bm U_i^{n,\sigma}=(u_i^{\sigma,(0)},u_i^{\sigma,(1)})^\top$ denote the pre-OE and OE-corrected moment vectors. Since $\sigma_i(u_h^*)$ is computed from the fixed reconstruction $u_h^*$ produced after the final RK stage, the OE process \eqref{eq4:1dOE} admits the explicit update
\begin{equation}\label{eq9:1dOE_sol}
u_i^{\sigma,(0)}=u_i^{n,(0)},
\qquad
u_i^{\sigma,(1)}=u_i^{n,(1)}\exp\left(-\frac{\beta_i\Delta t}{\Delta x}\sigma_i(u_h^*)\right).
\end{equation}
\end{proposition}
\begin{proof}
Set $\alpha_i=\beta_i\sigma_i(u_h^*)/\Delta x$, which is constant during the pseudo-time evolution because $u_h^*$ is frozen. Taking the cell average of \eqref{eq7:OE1d_ODE} shows that the average is unchanged, while taking the scaled first-order moment against $\phi_i^{(1)}$ yields the scalar ordinary differential equation (ODE) $\frac{\mathrm d}{\mathrm d\tau}u_i^{(1)}(\tau)=-\alpha_i u_i^{(1)}(\tau)$. Evaluating at $\tau=\Delta t$ gives \eqref{eq9:1dOE_sol}.
\end{proof}

Thus \eqref{eq9:1dOE_sol} uses only cell averages, scaled first-order moments, and derivative-free jumps of the reconstructed polynomial. The next proposition discusses conservation and high-order smallness in smooth regions.

\begin{proposition}
\label{prop:OEbasic}
\textbf{Key properties of the one-dimensional OE operator.} Assume the OE update is given by \eqref{eq9:1dOE_sol}.
\begin{enumerate}[label=(\roman*),leftmargin=18pt]
\item \textbf{Conservation and nonexpansive damping.} The OE step preserves the cell average and damps the scaled first-order moment nonexpansively:
\[
|u_i^{\sigma,(1)}|=e^{-\theta_i}|u_i^{n,(1)}|\le |u_i^{n,(1)}|,
\qquad
\theta_i:=\frac{\beta_i\Delta t}{\Delta x}\sigma_i(u_h^*)\ge0.
\]
\item \textbf{Accuracy preservation in smooth regions.} If the exact solution is smooth on $\Lambda_i$, the reconstruction is $k$th-order accurate on that stencil, $\Delta t=\mathcal O(\Delta x)$, and the active denominator in Table~\ref{tab:sigma1d} is either below $\epsilon$ or bounded below independently of $\Delta x$, then
\[
\jump{u_h^*}_{i\pm\frac12}=\mathcal O(\Delta x^k),
\qquad
u_i^{\sigma,(1)}-u_i^{n,(1)}=\mathcal O(\Delta x^{k+1}).
\]
Hence the OE correction does not reduce the designed spatial order.
\end{enumerate}
\end{proposition}
\begin{proof}
Item (i) follows directly from \eqref{eq9:1dOE_sol}. For item (ii), the two one-sided traces approximate the same smooth exact interface value with error $\mathcal O(\Delta x^k)$, so the jump is $\mathcal O(\Delta x^k)$. If the denominator in the definition of $\sigma_i$ falls below $\epsilon$, then the coefficient is zero. Otherwise, under the stated nondegeneracy assumption, the denominator is bounded away from zero and $\sigma_i=\mathcal O(\Delta x^k)$. Since $u_i^{n,(1)}=\mathcal O(\Delta x)$ for smooth solutions, the expansion $e^{-\theta_i}-1=-\theta_i+\mathcal O(\theta_i^2)$ yields
\[
u_i^{\sigma,(1)}-u_i^{n,(1)}=u_i^{n,(1)}\big(e^{-\theta_i}-1\big)=\mathcal O(\Delta x)\,\mathcal O(\Delta x^k)=\mathcal O(\Delta x^{k+1}).
\]
\end{proof}
\begin{remark}
If $u_p=c+\varepsilon v_p$ on $\Lambda_i$ with large background $|c|$ and small fluctuation $|\varepsilon v_p|$, then the Type~II denominator satisfies $\max_{p\in\Lambda_i}|u_p-S_{\Lambda_i}|=\mathcal O(|c|)$, weakening the damping of the fluctuation. This explains its use in the multiscale tests; Type~I is retained when global variation is the intended normalization scale.
\end{remark}

\subsection{Cartesian two-dimensional extension}\label{sec:FAM_2d}
In this work, we formulate the two-dimensional extension on uniform Cartesian grids. Consider the two-dimensional scalar conservation law
\begin{equation}\label{eq10:2dHCLS}
\partial_t u+\partial_x f_1(u)+\partial_y f_2(u)=0,
\qquad
(x,y,t)\in\Omega\times[0,T].
\end{equation}
On a rectangular cell $I_{i,j}=[x_{i-\frac12},x_{i+\frac12}]\times[y_{j-\frac12},y_{j+\frac12}]$, let $I_i=[x_{i-\frac12},x_{i+\frac12}]$, $I_j=[y_{j-\frac12},y_{j+\frac12}]$, $\Delta x=x_{i+\frac12}-x_{i-\frac12}$, $\Delta y=y_{j+\frac12}-y_{j-\frac12}$, and $(x_i,y_j)$ be the cell center. We evolve the cell average and the two scaled first-order moments,
\[
\phi_{i,j}^{(0)}(x,y)=1,
\qquad
\phi_{i,j}^{(1)}(x,y)=\frac{x-x_i}{\Delta x},
\qquad
\phi_{i,j}^{(2)}(x,y)=\frac{y-y_j}{\Delta y},
\]
\[
u_{i,j}^{(l)}(t)\approx \frac{1}{\Delta x\Delta y}\iint_{I_{i,j}}u(x,y,t)\,\phi_{i,j}^{(l)}(x,y)\,\mathrm dx\mathrm dy,
\qquad l=0,1,2.
\]
The corresponding Cartesian weak form is
\begin{equation}\label{eq11:2dSemi}
\frac{\mathrm d}{\mathrm dt}u_{i,j}^{(l)}(t)
=\frac{1}{\Delta x\Delta y}
\left(
\iint_{I_{i,j}}\bm f(u_h)\cdot\nabla\phi_{i,j}^{(l)}\,\mathrm dx\mathrm dy
-\int_{\partial I_{i,j}}\widehat{\bm f\cdot\bm n}\,\phi_{i,j}^{(l)}\,\mathrm ds
\right),
\qquad l=0,1,2,
\end{equation}
where $\bm f=(f_1,f_2)$ and $\bm n$ is the outward normal unit vector.
With the SSP-RK3 method, one time step takes the form
\begin{equation}\label{eq12:RK}
\begin{cases}
{\bm U}^{n,1}_{i,j}={\bm U}^{n,0}_{i,j}+\Delta t\,\bm{\mathcal R}_{i,j}(u_h^{n,0}),\\
{\bm U}^{n,2}_{i,j}=\dfrac34{\bm U}^{n,0}_{i,j}+\dfrac14\Big({\bm U}^{n,1}_{i,j}+\Delta t\,\bm{\mathcal R}_{i,j}(u_h^{n,1})\Big),\\
{\bm U}^{n}_{i,j}=\dfrac13{\bm U}^{n,0}_{i,j}+\dfrac23\Big({\bm U}^{n,2}_{i,j}+\Delta t\,\bm{\mathcal R}_{i,j}(u_h^{n,2})\Big),\\
{\bm U}^{n,\sigma}_{i,j}=\mathcal F_{\rm OE}\{{\bm U}_\kappa^n\}_{\kappa\in\mathcal{S}_{i,j}}, 
\end{cases}
\end{equation}
where ${\bm U}_{i,j}^{n,0}={\bm U}_{i,j}^{n-1,\sigma}$, $\mathcal{S}_{i,j}=\bigcup^{1}_{p,q=-1}\Lambda_{i+p,j+q}$, $\Lambda_{i,j}=\{(i+s,j+\ell):s,\ell\in\{-1,0,1\}\}$, and ${\bm U}_{i,j}=\big(u_{i,j}^{(0)},u_{i,j}^{(1)},u_{i,j}^{(2)}\big)^\top$.

The rectangular-grid OE update keeps the cell average and damps the two scaled first-order moments with a common exponential factor,
\begin{equation}\label{eq:2dOE_sol}
(u_{i,j}^{\sigma,(0)},u_{i,j}^{\sigma,(1)},u_{i,j}^{\sigma,(2)})=
(u_{i,j}^{n,(0)},\delta_{i,j}u_{i,j}^{n,(1)},\delta_{i,j}u_{i,j}^{n,(2)}),
\end{equation}
with
\[
\delta_{i,j}=\exp\left(-\frac{\beta_{i,j}^x\Delta t}{\Delta x}\sigma_{i,j}^x(u_h^*)-\frac{\beta_{i,j}^y\Delta t}{\Delta y}\sigma_{i,j}^y(u_h^*)\right),
\]
where $\beta_{i,j}^x=\max_{\kappa\in\Lambda_{i,j}}|f_1'(u_\kappa^{n,(0)})|$ and $\beta_{i,j}^y=\max_{\kappa\in\Lambda_{i,j}}|f_2'(u_\kappa^{n,(0)})|$ are local wave-speed estimates in the $x$- and $y$-directions. The jumps across the vertical and horizontal interfaces are
\[
\jump{u_h^*}_{i+\frac12,j}=u_h^*(x_{i+\frac12}^+,y_j)-u_h^*(x_{i+\frac12}^-,y_j),
\qquad
\jump{u_h^*}_{i,j+\frac12}=u_h^*(x_i,y_{j+\frac12}^+)-u_h^*(x_i,y_{j+\frac12}^-),
\]
and we abbreviate their directional accumulations by
\[
J_{i,j}^x:=|\jump{u_h^*}_{i-\frac12,j}|+|\jump{u_h^*}_{i+\frac12,j}|,
\qquad
J_{i,j}^y:=|\jump{u_h^*}_{i,j-\frac12}|+|\jump{u_h^*}_{i,j+\frac12}|.
\]
We further introduce the global average and local reference sum
\[
\bar u_\Omega:=\frac{1}{N_xN_y}\sum_{p=1}^{N_x}\sum_{q=1}^{N_y}u_{p,q}^{n,(0)},
\qquad
S_{\Lambda_{i,j}}:=\sum_{(p,q)\in\Lambda_{i,j}}u_{p,q}^{n,(0)},
\]
together with the corresponding denominators
\[
D_\Omega:=\max_{p,q}|u_{p,q}^{n,(0)}-\bar u_\Omega|,
\qquad
D_{i,j}:=\max_{(p,q)\in\Lambda_{i,j}}|u_{p,q}^{n,(0)}-S_{\Lambda_{i,j}}|.
\]
The two Cartesian OE coefficients are then defined as in Table~\ref{tab:sigma2d}.

\begin{table}[t]
\centering
\caption{Two definitions of the Cartesian damping coefficients $(\sigma_{i,j}^x,\sigma_{i,j}^y)$ used in the two-dimensional OE update.}
\label{tab:sigma2d}
\small
\begin{tabular}{c|l|l}
	\toprule[1.2pt]
	Coefficient & Type I &Type II \\
	\midrule
	$(\sigma_{i,j}^x,\sigma_{i,j}^y)$  &
	$\displaystyle
	\left\{
	\begin{aligned}
		&(0,0), &&\text{if }D_\Omega<\epsilon_\Omega,\\
		&\left(\gamma_k\frac{J_{i,j}^x}{D_\Omega},\,\gamma_k\frac{J_{i,j}^y}{D_\Omega}\right), &&\text{otherwise},
	\end{aligned}
	\right.$ 
	 & 
	$\displaystyle
	\left\{
	\begin{aligned}
		&(0,0), &&\text{if }D_{i,j}<\epsilon_{i,j},\\
		&\left(\gamma_k\frac{J_{i,j}^x}{D_{i,j}},\,\gamma_k\frac{J_{i,j}^y}{D_{i,j}}\right), &&\text{otherwise},
	\end{aligned}
	\right.$ \\
	\bottomrule[1.2pt]
\end{tabular}
\end{table}

In Table~\ref{tab:sigma2d}, the zero-denominator tolerance is chosen analogously,
\[
\epsilon_\Omega=10^{-14}\max\{\max_{p,q}|u_{p,q}^{n,(0)}|,1\},
\qquad
\epsilon_{i,j}=10^{-14}\max\{\max_{(p,q)\in\Lambda_{i,j}}|u_{p,q}^{n,(0)}|,1\},
\]
for Type~I and Type~II, respectively. As in one dimension, the Type~II reference quantity is intentionally unnormalized, making it less dissipative for small-amplitude structures on a large background. We use Type~II in the multiscale tests and Type~I when the global variation is the intended scale. In two dimensions, $\gamma_3=\gamma_4=25$ and $\gamma_5=\gamma_6=15$; the supplementary sensitivity plots show only mild dependence on nearby values.

\begin{proposition}
\label{prop:2dOEbasic}
\textbf{Key properties of the Cartesian OE operator.} Assume the OE update is given by \eqref{eq:2dOE_sol}.
\begin{enumerate}[label=(\roman*),leftmargin=18pt]
\item \textbf{Conservation and common nonexpansive damping.} The OE step preserves the cell average and damps both scaled first-order moments with the same factor:
\[
u_{i,j}^{\sigma,(0)}=u_{i,j}^{n,(0)},
\qquad
|u_{i,j}^{\sigma,(r)}|=\delta_{i,j}|u_{i,j}^{n,(r)}|\le |u_{i,j}^{n,(r)}|,
\quad r=1,2,
\]
with $0<\delta_{i,j}\le 1$.
\item \textbf{Accuracy preservation in smooth regions.} Let $h\sim\max\{\Delta x,\Delta y\}$ on the uniform Cartesian meshes. If the exact solution is smooth on $\Lambda_{i,j}$, the reconstruction is $k$th-order accurate at the face-center points used in the OE jumps and at the quadrature points entering \eqref{eq11:2dSemi}, $\Delta t=\mathcal O(h)$, and the active denominator in Table~\ref{tab:sigma2d} is either below $\epsilon$ or bounded below independently of $h$, then
\[
J_{i,j}^x=\mathcal O(h^k),
\qquad
J_{i,j}^y=\mathcal O(h^k),
\qquad
u_{i,j}^{\sigma,(r)}-u_{i,j}^{n,(r)}=\mathcal O(h^{k+1}),
\quad r=1,2.
\]
Hence the Cartesian OE correction does not reduce the designed spatial order.
\end{enumerate}
\end{proposition}
\begin{proof}
Item (i) follows directly from \eqref{eq:2dOE_sol}. For item (ii), each pair of one-sided traces at a face approximates the same smooth exact value with error $\mathcal O(h^k)$, so the directional jump sums are $\mathcal O(h^k)$. If the denominator in the chosen coefficient falls below $\epsilon$, the coefficient is zero. Otherwise, under the stated nondegeneracy assumption, it is bounded away from zero, whence $\sigma_{i,j}^x(u_h^*)=\mathcal O(h^k)$ and $\sigma_{i,j}^y(u_h^*)=\mathcal O(h^k)$. Since the scaled first-order moments satisfy $u_{i,j}^{n,(1)}=\mathcal O(h)$ and $u_{i,j}^{n,(2)}=\mathcal O(h)$ for smooth solutions, the expansion $\delta_{i,j}-1=\mathcal O(h^k)$ yields the stated $\mathcal O(h^{k+1})$ correction.
\end{proof}

For $k=3,4,5,6$ these jumps can be written explicitly in the compact matrix form
\begin{equation}\label{eq:2Djump}
\left\{
\begin{aligned}
\jump{u_h^*}_{i+\frac12,j}&=
\left\langle \bfA_k,{\bf U}^{n,(0)}_x\right\rangle+
\left\langle \bfB_k,{\bf U}^{n,(1)}_x\right\rangle+
\left\langle \bfC_k,{\bf U}^{n,(2)}_x\right\rangle,\\
\jump{u_h^*}_{i,j+\frac12}&=
\left\langle \bfA_k,({\bf U}^{n,(0)}_y)^\top\right\rangle+
\left\langle \bfC_k,({\bf U}^{n,(1)}_y)^\top\right\rangle+
\left\langle \bfB_k,({\bf U}^{n,(2)}_y)^\top\right\rangle,
\end{aligned}
\right.
\end{equation}
where $\langle\cdot,\cdot\rangle$ denotes the Frobenius inner product, ${\bf U}^{n,(r)}_x=[u_{i+s,j+\ell}^{n,(r)}]\in\mathbb R^{4\times3}$ and ${\bf U}^{n,(r)}_y=[u_{i+\ell,j+s}^{n,(r)}]\in\mathbb R^{3\times4}$ collect the stencil moments with $r=0,1,2$, $\ell=-1,0,1$, and $s=-1,0,1,2$, and the constant matrices $\bfA_k$, $\bfB_k$, and $\bfC_k$ are listed in tabular form in the supplementary material.

\subsection{Systems and componentwise implementation}\label{sec:FAM_system}

We now consider the hyperbolic system $\partial_t\bm u+\nabla\cdot\bm f(\bm u)=0$ with $\bm u\in\mathbb R^N$. In the proposed FAM formulation, RK-stage reconstructions remain linear and are applied componentwise to conservative variables. The OE damping is supplied by the single end-of-step correction; the standard bound-preserving/positivity-preserving (BP/PP) limiter is activated only when explicitly identified in the benchmark descriptions and in the separately noted no-OE ablation variant.

The cost profile follows from two structural features: the evolved state is fixed at the $\mathbb P^1$ moment level, and higher order changes only the precomputed linear reconstruction operator. 

\begin{remark}[Fixed linear reconstructions and variable transformations]
Each RK-stage reconstruction used by FAM is a fixed linear operator on the stencil moments. It therefore commutes with any \emph{fixed} linear change of variables: reconstructing in transformed variables and mapping back yields the same polynomial as reconstructing componentwise in conservative variables. The relevant issue is thus not the fixed transform itself, but the cost and possible oscillation-control benefit of state-dependent local characteristic-variable transformations. Section~\ref{sec:NumTest} will compare the conservative-variable and local characteristic-variable FAM5 implementations. 
\end{remark}

At time $t^n$, let $\bm U_K^n$ denote the FAM solution after the final RK stage and before the OE process, and let $\bm U_K^{n,\sigma}$ represent the OE-corrected moments. We use $\varrho(\cdot)$ to denote the spectral radius of a matrix, consistent with the Rusanov flux used in all numerical tests. In one dimension, with $\bm U_i^n=({\bm u}_i^{n,(0)},{\bm u}_i^{n,(1)})^\top$, the corresponding OE process $\bm U_i^{n,\sigma}=\mathcal F_{\rm OE}\{\bm U_\kappa^n\}_{\kappa\in\Lambda_i}$ can be expressed componentwise as
\begin{equation}\label{eq:1dOE_system}
{\bm u}_i^{\sigma,(0)}={\bm u}_i^{n,(0)},
\qquad
{\bm u}_i^{\sigma,(1)}={\bm u}_i^{n,(1)}\exp\left(-\beta_i\frac{\Delta t}{\Delta x}\sigma_i({\bm u}_h^*)\right),
\end{equation}
where
\[
\beta_i:=\max_{\kappa\in\Lambda_i}\varrho\!\left(\frac{\partial \bm f}{\partial \bm u}({\bm u}_\kappa^{n,(0)})\right),
\qquad
\sigma_i({\bm u}_h^*)=\max_{1\le \ell\le N}\widehat\sigma_i(u_h^{*,\ell}),
\]
and $\widehat\sigma_i$ is the scalar coefficient from Table~\ref{tab:sigma1d}. For two-dimensional systems we write $\bm U_{i,j}^n=({\bm u}_{i,j}^{n,(0)},{\bm u}_{i,j}^{n,(1)},{\bm u}_{i,j}^{n,(2)})^\top$. The OE process $\bm U_{i,j}^{n,\sigma}=\mathcal F_{\rm OE}\{\bm U_\kappa^n\}_{\kappa\in\Lambda_{i,j}}$ becomes
\begin{equation}\label{eq:2dOE_system}
{\bm u}_{i,j}^{\sigma,(0)}={\bm u}_{i,j}^{n,(0)},
\qquad
{\bm u}_{i,j}^{\sigma,(1)}={\bm u}_{i,j}^{n,(1)}\delta_{i,j},
\qquad
{\bm u}_{i,j}^{\sigma,(2)}={\bm u}_{i,j}^{n,(2)}\delta_{i,j},
\end{equation}
with
\[
\delta_{i,j}=\exp\left(-\frac{\beta_{i,j}^x\Delta t}{\Delta x}\sigma^{x,\rm sys}_{i,j}({\bm u}_h^*)-\frac{\beta_{i,j}^y\Delta t}{\Delta y}\sigma^{y,\rm sys}_{i,j}({\bm u}_h^*)\right),
\]
where
\[
\beta_{i,j}^x:=\max_{\kappa\in\Lambda_{i,j}}\varrho\!\left(\frac{\partial \bm f_1}{\partial \bm u}({\bm u}_\kappa^{n,(0)})\right),
\qquad
\beta_{i,j}^y:=\max_{\kappa\in\Lambda_{i,j}}\varrho\!\left(\frac{\partial \bm f_2}{\partial \bm u}({\bm u}_\kappa^{n,(0)})\right),
\]
\[
\sigma^{x,\rm sys}_{i,j}({\bm u}_h^*)=\max_{1\le \ell\le N}\sigma_{i,j}^x(u_h^{*,\ell}),
\qquad
\sigma^{y,\rm sys}_{i,j}({\bm u}_h^*)=\max_{1\le \ell\le N}\sigma_{i,j}^y(u_h^{*,\ell}).
\]

Hence the system OE step preserves conservative cell averages and damps each scaled first-order moment vector by a common scalar factor assembled from componentwise jump indicators. This is the mechanism used in all the tested cases.

The BP/PP limiter can be used in conjunction with the OE process. When active in our tested cases, the BP/PP limiter uses the optimal cell-average decomposition strategy of \cite{CDW1,CDW2}, as in \cite{FanWu}. This limiter is separate from OE. 

\begin{algorithm}[t]
\caption{One time step of the 1D FAM$k$ schemes with the end-of-step OE correction and optional BP/PP safeguard}
\label{alg:1dFAM}
\begin{algorithmic}[1]
\State Given $\bm U_i^{n,0}=\bm U_i^{n-1,\sigma}$, $i=1,\dots,N_x$.
\For{each SSP-RK substep}
\State Reconstruct the current stage solution componentwise by the degree-$(k-1)$ FAM polynomial; see Subsection~\ref{sec:FAM_HLR}.
\State Update the moments by the corresponding Shu--Osher formula in \eqref{eq3:1dRK}.
\State After the substep, apply the auxiliary BP/PP limiter only if it is activated for the current test; this is a separate admissibility safeguard, not part of the OE definition.
\EndFor
\State Compute $\bm U_i^{n,\sigma}$ from $\bm U_i^n$ by the OE update \eqref{eq9:1dOE_sol} (scalar) or \eqref{eq:1dOE_system} (systems).
\end{algorithmic}
\end{algorithm}

\begin{remark}
In our computations, we use the Lax--Friedrichs (i.e., Rusanov) numerical flux with the same wave-speed convention across the compared methods. Volume and face integrals are evaluated by four-point Gauss quadrature in each coordinate direction, whose formal order is adequate for the polynomial degrees used in the reported tests. Periodic, reflective, and outflow ghost cells treat the scaled first-order moments consistently with the cell averages; for inflow boundaries we set the scaled first-order moments in ghost cells to zero.
\end{remark}

\subsection{Compact linear reconstructions and consistency}\label{sec:FAM_HLR}

The compact reconstructions match available cell averages and scaled first-order moments on compact stencils, solve the resulting square systems in orthogonal polynomial bases, and retain the terms needed for the target consistency order. Explicit coefficients, a worked system, and jump matrices are deferred to the supplementary material.

The design is paired: in one dimension, FAM3/FAM4 share a cubic auxiliary polynomial and FAM5/FAM6 share a quintic one; on Cartesian meshes, the corresponding pairs share incomplete quartic and degree-six auxiliary ansatzes. This paired structure lets all four schemes reuse the same fixed-moment evolved state.

\paragraph{One-dimensional reconstructions}
For the 3rd- and 4th-order schemes, we construct a cubic polynomial $Q(x)$ satisfying
\begin{equation}\label{eq5:Qx}
\frac{1}{\Delta x}\int_{I_\ell}Q(x)\,\mathrm dx=u_\ell^{(0)},\quad \ell=i-1,i,i+1,
\qquad
\frac{1}{\Delta x}\int_{I_i}Q(x)\frac{x-x_i}{\Delta x}\,\mathrm dx=u_i^{(1)}.
\end{equation}
Writing $Q(x)=\sum_{\ell=0}^{3}c_\ell\phi_i^{(\ell)}(x)$ in the local orthogonal basis collected in the supplementary material, the coefficients $c_\ell$ follow from a $4\times4$ linear system. The 3rd-order reconstruction is then the quadratic projection $p_2(x)=\sum_{\ell=0}^{2}c_\ell\phi_i^{(\ell)}(x)$, while the 4th-order reconstruction is $p_3(x)=Q(x)$.

For the 5th- and 6th-order schemes, we construct a quintic polynomial $Q(x)$ satisfying
\begin{equation}\label{eq6:Qx}
\frac{1}{\Delta x}\int_{I_\ell}Q(x)\,\mathrm dx=u_\ell^{(0)},
\qquad
\frac{1}{\Delta x}\int_{I_\ell}Q(x)\frac{x-x_\ell}{\Delta x}\,\mathrm dx=u_\ell^{(1)},
\qquad
\ell=i-1,i,i+1.
\end{equation}
Here the coefficients of $Q(x)=\sum_{\ell=0}^{5}c_\ell\phi_i^{(\ell)}(x)$ are obtained from a $6\times6$ moment-matching system. The 5th-order reconstruction is $p_4(x)=\sum_{\ell=0}^{4}c_\ell\phi_i^{(\ell)}(x)$ and the 6th-order reconstruction is $p_5(x)=Q(x)$.

\begin{proposition}
\label{prop:1d-recon}
\textbf{One-dimensional reconstruction exactness and consistency.} Let $\mathcal R_i^{(k)}u$ denote the degree-$(k-1)$ polynomial returned by the FAM$k$ reconstruction on $I_i$ from the exact stencil moments of a smooth function $u$.
\begin{enumerate}[label=(\roman*),leftmargin=18pt]
\item The cubic and quintic moment-matching systems are uniquely solvable for arbitrary admissible stencil moments.
\item The reconstructions reproduce polynomials of the target degree exactly: FAM3 and FAM4 are exact on $\mathbb P^2$ and $\mathbb P^3$, respectively, while FAM5 and FAM6 are exact on $\mathbb P^4$ and $\mathbb P^5$, respectively.
\item If $u\in C^k$ on the reconstruction stencil, then
\[
\|u-\mathcal R_i^{(k)}u\|_{L^\infty(I_i)}+|u(x_{i\pm\frac12})-(\mathcal R_i^{(k)}u)(x_{i\pm\frac12})|=\mathcal O(\Delta x^k).
\]
Consequently, the reconstructed traces and the volume term in \eqref{eq:1dSemi} are $k$th-order consistent.
\end{enumerate}
\end{proposition}
\begin{proof}
The supplementary material gives explicit coefficient formulas showing that the cubic and quintic coefficients are linear combinations of the prescribed moments; hence the corresponding square systems are uniquely solvable. Exactness follows because if $q$ already belongs to the target polynomial space, then its exact stencil moments satisfy the defining constraints; uniqueness therefore implies that the recovered polynomial is $q$ itself. The stated $\mathcal O(\Delta x^k)$ consistency estimate then follows from a Taylor expansion of $u$ about $x_i$: the reconstruction operator is linear, acts on a fixed uniform stencil, and is exact on all polynomials up to degree $k-1$, so the first neglected term is of order $\Delta x^k$ both in the cell interior and at the interface points.
\end{proof}

\paragraph{Two-dimensional reconstructions}
On a $3\times3$ Cartesian patch, we relabel the central cell and its neighbors as $I_{i-1,j-1}\equiv I_1, I_{i,j-1}\equiv I_2, I_{i+1,j-1}\equiv I_3, I_{i-1,j}\equiv I_4, I_{i,j}\equiv I_5, I_{i+1,j}\equiv I_6, I_{i-1,j+1}\equiv I_7, I_{i,j+1}\equiv I_8, I_{i+1,j+1}\equiv I_9.$ Correspondingly, the moments are denoted as $u^{(l)}_{1},\ldots,u^{(l)}_{9} $ for $l=0,1,2$. Let $\mathbb P_m=\operatorname{span}\{x^a y^b:\ a,b\ge 0,\ a+b\le m\}$ be the complete polynomial space of total degree at most $m$, and let $\Pi_m$ denote the projection onto $\mathbb P_m$.

For the 3rd- and 4th-order schemes, we first construct an auxiliary polynomial $Q(x,y)$ in an incomplete quartic space of dimension $11$. This auxiliary polynomial is determined by the following $11$ moment constraints:
\begin{equation}\label{Qxy_quartic}
\left\{
\begin{aligned}
&\frac{1}{\Delta x\Delta y}\int_{I_\ell}Q(x,y)\,\mathrm dx\mathrm dy=u_\ell^{(0)},\qquad \ell=1,\dots,9,\\
&\frac{1}{\Delta x\Delta y}\int_{I_5}Q(x,y)\frac{x-x_5}{\Delta x}\,\mathrm dx\mathrm dy=u_5^{(1)},\\
&\frac{1}{\Delta x\Delta y}\int_{I_5}Q(x,y)\frac{y-y_5}{\Delta y}\,\mathrm dx\mathrm dy=u_5^{(2)}.
\end{aligned}
\right.
\end{equation}
In an orthogonal basis, these $11$ constraints determine the $11$ coefficients of $Q(x,y)$ uniquely. The final reconstruction polynomials are then obtained by projecting the auxiliary polynomial onto complete polynomial spaces: $p_{3}(x,y)=\Pi_2 Q(x,y), p_{4}(x,y)=\Pi_3 Q(x,y).$ Thus, FAM3 uses a complete quadratic polynomial, while FAM4 uses a complete cubic polynomial. The incomplete quartic space is used only at the auxiliary stage to close the square moment system.

For the 5th- and 6th-order schemes, we construct an auxiliary polynomial $Q(x,y)$ in an incomplete degree-six space of dimension $23$. This auxiliary polynomial is determined by the following $23$ moment constraints:
\begin{equation}\label{Qxy_quintic}
\left\{
\begin{aligned}
&\frac{1}{\Delta x\Delta y}\int_{I_\ell}Q(x,y)\,\mathrm dx\mathrm dy=u_\ell^{(0)},\qquad \ell=1,\dots,9,\\
&\frac{1}{\Delta x\Delta y}\int_{I_\ell}Q(x,y)\frac{x-x_\ell}{\Delta x}\,\mathrm dx\mathrm dy=u_\ell^{(1)},\qquad \ell\in\{1,3,4,5,6,7,9\},\\
&\frac{1}{\Delta x\Delta y}\int_{I_\ell}Q(x,y)\frac{y-y_\ell}{\Delta y}\,\mathrm dx\mathrm dy=u_\ell^{(2)},\qquad \ell\in\{1,2,3,5,7,8,9\}.
\end{aligned}
\right.
\end{equation}
The resulting square system determines $Q(x,y)$ uniquely. The final reconstruction polynomials are obtained by projecting $Q$ onto complete polynomial spaces:  $p_{5}(x,y)=\Pi_4 Q(x,y), p_{6}(x,y)=\Pi_5 Q(x,y).$ Therefore, FAM5 uses a complete quartic polynomial, while FAM6 uses a complete quintic polynomial. The incomplete degree-six space is used only at the auxiliary stage to close the square moment system. Detailed coefficients are provided in the supplementary material.

Note that $Q(x,y)$ is not used directly as the final reconstruction polynomial; it serves as an auxiliary polynomial from which the complete-degree reconstructions are extracted by projection.

\begin{proposition}
\label{prop:2d-recon}
\textbf{Cartesian two-dimensional reconstruction exactness and consistency.} Let $\mathcal R_{i,j}^{(k)}u$ denote the polynomial returned by the Cartesian FAM$k$ reconstruction from the exact cell averages and scaled first-order moments on the $3\times3$ patch around $I_{i,j}$.
\begin{enumerate}[label=(\roman*),leftmargin=18pt]
\item The linear systems associated with \eqref{Qxy_quartic} and \eqref{Qxy_quintic} are uniquely solvable on the chosen incomplete polynomial bases.
\item FAM3 and FAM4 reproduce polynomials of total degree at most $2$ and $3$, respectively, while FAM5 and FAM6 reproduce polynomials of total degree at most $4$ and $5$, respectively.
\item If $u\in C^k$ on the patch, then the reconstructed polynomial is $k$th-order consistent in the cell, at the face-center locations used by the OE jumps, and at the Cartesian face quadrature points used by the method. In particular, the face traces and volume terms entering \eqref{eq11:2dSemi} are $k$th-order accurate.
\end{enumerate}
\end{proposition}
\begin{proof}
The supplementary material gives explicit coefficient formulas for the incomplete quartic and incomplete sixth-degree ansatzes, so the corresponding square systems are uniquely solvable. Exactness is proved as in Proposition~\ref{prop:1d-recon}: if $q$ already lies in the target polynomial space, then its exact patch moments satisfy the defining constraints, and uniqueness forces the recovered polynomial to equal $q$. Since the operator is linear, local, and exact on all polynomials through degree $k-1$, a two-variable Taylor expansion about $(x_i,y_j)$ gives the stated $\mathcal O(h^k)$ consistency, with $h\sim \max\{\Delta x,\Delta y\}$ on the uniform Cartesian meshes.
\end{proof}

\section{Fourier analysis for the linear-advection model problem}\label{sec:FAM_SaE}

To isolate the behavior of the underlying linear reconstruction, this section studies the one-dimensional FAM schemes \emph{without} the OE step. We consider the linear advection equation
\begin{equation}\label{eq:1dlinear}
\partial_t u+\partial_x(au)=0,
\qquad
x\in[0,2\pi],\quad t>0,
\qquad
u(x,0)=\hat u_0e^{\mathrm i\omega x},
\end{equation}
with constant $a>0$ and periodic boundary conditions.

With upwind flux, the semi-discrete FAM$k$ scheme can be written as
\begin{equation}\label{eq:linear-semi}
\frac{\mathrm d\bm u_i(t)}{\mathrm dt}
=\frac{a}{\Delta x}
\left(\bm D_{-2}^{(k)}\bm u_{i-2}(t)+\bm D_{-1}^{(k)}\bm u_{i-1}(t)+\bm D_{0}^{(k)}\bm u_i(t)+\bm D_{1}^{(k)}\bm u_{i+1}(t)\right),
\end{equation}
where $\bm u_i=(u_i^{(0)},u_i^{(1)})^\top$ and the constant $2\times2$ matrices $\bm D_{\ell}^{(k)}$ are listed in the supplementary material. Using the Fourier ansatz $\bm u_i(t)=\hat{\bm u}(t)e^{\mathrm i\omega x_i}$ yields
\begin{equation}\label{eq:uhat}
\frac{\mathrm d\hat{\bm u}(t)}{\mathrm dt}=a\widetilde{\bm G}\,\hat{\bm u}(t),
\end{equation}
with symbol
\begin{equation}\label{eq:Gmatrix}
\widetilde{\bm G}
=\frac{1}{\Delta x}
\left(
\bm D_{-2}^{(k)}e^{-2\mathrm i\xi}+\bm D_{-1}^{(k)}e^{-\mathrm i\xi}+\bm D_0^{(k)}+\bm D_1^{(k)}e^{\mathrm i\xi}
\right),
\qquad
\xi=\omega\Delta x\in[0,2\pi].
\end{equation}
We denote by $\hat{\bm u}_{\rm ex}$ the exact vector of the cell average and scaled first-order moment associated with the Fourier mode in \eqref{eq:1dlinear}. With the normalization induced by the ansatz $\bm u_i(t)=\hat{\bm u}(t)e^{\mathrm i\omega x_i}$, it is given by
\[
\hat{\bm u}_{\rm ex}=\hat u_0
\begin{pmatrix}
\dfrac{2\sin(\xi/2)}{\xi}, 
\dfrac{\mathrm i\big(2\sin(\xi/2)-\xi\cos(\xi/2)\big)}{\xi^2}
\end{pmatrix}^\top,
\qquad \xi=\omega\Delta x,
\]
with the continuous value $\hat{\bm u}_{\rm ex}=\hat u_0(1,0)^\top$ at $\xi=0$.

\begin{proposition}
\label{prop:semi}
For $k=3,4,5,6$, the symbol $\widetilde{\bm G}$ is diagonalizable for sufficiently small $\Delta x$ and has two distinct eigenvalues $\zeta_1(\omega,\Delta x)$ and $\zeta_2(\omega,\Delta x)$ such that
\[
\zeta_1=-\mathrm i\omega+\mathcal O(\Delta x^{k+1}),
\qquad
\Re(\zeta_2)\le -\frac{c}{\Delta x}
\]
for some $c>0$ independent of $\Delta x$. Moreover, if $\bm V_1$ and $\bm V_2$ denote the corresponding modal contributions to the solution of \eqref{eq:uhat}, then
\[
\|\bm V_1-\hat{\bm u}_{\rm ex}\|=\mathcal O(\Delta x^k),
\qquad
\|\bm V_2\|=\mathcal O(\Delta x^k)
\]
for any vector norm, where $\hat{\bm u}_{\rm ex}$ is the exact Fourier moment vector associated with the initial mode.
\end{proposition}
\begin{proof}
For each $k$, we expand the characteristic polynomial of $\widetilde{\bm G}$ symbolically in powers of $\Delta x$. The full asymptotic expansions are collected in the supplementary material. Those expansions show that one eigenvalue satisfies $\zeta_1=-\mathrm i\omega+\mathcal O(\Delta x^{k+1})$, while the other satisfies $\Re(\zeta_2)\le-c/\Delta x$ for some $c>0$. Let $\bm r_1$ and $\bm r_2$ be corresponding right eigenvectors, and decompose the exact initial moment vector as
$
\hat{\bm u}_{\rm ex}=\alpha_1\bm r_1+\alpha_2\bm r_2.
$ 
The same symbolic expansions show that $\bm r_1=\hat{\bm u}_{\rm ex}+\mathcal O(\Delta x^k)$, $\alpha_1=1+\mathcal O(\Delta x^k)$, and $\alpha_2=\mathcal O(\Delta x^k)$. Defining $\bm V_1:=\alpha_1\bm r_1$ and $\bm V_2:=\alpha_2\bm r_2$ therefore gives
\[
\|\bm V_1-\hat{\bm u}_{\rm ex}\|=\mathcal O(\Delta x^k),
\qquad
\|\bm V_2\|=\mathcal O(\Delta x^k),
\]
which is the stated modal decomposition.
\end{proof}

The full symbolic expansions of the physical and spurious modes for $k=3,4,5,6$ are collected in the supplementary material.

\begin{proposition}
\label{prop:order}
Under the assumptions of Proposition~\ref{prop:semi}, let $\bm u(T)$ and $\bm u_h(T)$ denote the exact and numerical moment vectors at time $T>0$. Then
\[
\|\bm u(T)-\bm u_h(T)\|\le C_1aT\,\Delta x^k+C_2\Delta x^k,
\]
where $C_1$ and $C_2$ are independent of $\Delta x$. Hence FAM$k$ is $k$th-order accurate in any vector norm.
\end{proposition}
\begin{proof}
Let $\hat{\bm u}_h(t)$ denote the numerical Fourier moment vector. By Proposition~\ref{prop:semi},
\[
\hat{\bm u}_h(t)=e^{a\zeta_1 t}\bm V_1+e^{a\zeta_2 t}\bm V_2.
\]
The exact Fourier moment vector evolves as $\hat{\bm u}(t)=e^{-\mathrm i a\omega t}\hat{\bm u}_{\rm ex}$. Hence
\[
\hat{\bm u}(t)-\hat{\bm u}_h(t)
=\big(e^{-\mathrm i a\omega t}-e^{a\zeta_1 t}\big)\bm V_1
+e^{-\mathrm i a\omega t}(\hat{\bm u}_{\rm ex}-\bm V_1)
-e^{a\zeta_2 t}\bm V_2.
\]
Because $\zeta_1+\mathrm i\omega=\mathcal O(\Delta x^{k+1})$, the first term is bounded by $C_1 a t\,\Delta x^k$ for $0\le t\le T$. The second term is $\mathcal O(\Delta x^k)$ by Proposition~\ref{prop:semi}. For the third term, $\Re(\zeta_2)\le-c/\Delta x$ and $\|\bm V_2\|=\mathcal O(\Delta x^k)$ imply
\[
\|e^{a\zeta_2 t}\bm V_2\|\le C_2\Delta x^k e^{-ca t/\Delta x}\le C_2\Delta x^k.
\]
Summing the three bounds at $t=T$ yields the stated estimate.
\end{proof}

\begin{proposition}
\label{prop:superavg}
\textbf{Cell-average superconvergence of the linear backbone.} Under the assumptions of Proposition~\ref{prop:semi}, if $\bm u(T)-\bm u_h(T)=(e^{(0)}(T),e^{(1)}(T))^\top$, then
\[
|e^{(0)}(T)|\le C_0\Delta x^{k+1},
\qquad
|e^{(1)}(T)|\le C_1\Delta x^k,
\]
for constants $C_0$ and $C_1$ independent of $\Delta x$.
\end{proposition}
\begin{proof}
The componentwise expansions collected in the supplementary material show that, for $k=3,4,5,6$, the first components of both $\bm V_1-\hat{\bm u}_{\rm ex}$ and $\bm V_2$ are $\mathcal O(\Delta x^{k+1})$, whereas their second components are $\mathcal O(\Delta x^k)$. Inserting these componentwise estimates into the proof of Proposition~\ref{prop:order} yields the stated orders for the two entries of the moment vector.
\end{proof}

\begin{proposition}
\label{prop:CFL}
Combining \eqref{eq:linear-semi} with the third-order SSP-RK discretization and sampling $10^4$ uniformly spaced wave numbers in $\xi\in[0,2\pi]$ yields the following sampled admissible CFL estimates for the fully discrete linear-backbone FAM schemes without OE:
\[
\begin{aligned}
&0<C_{\rm CFL}\le 0.40\ (k=3),
\qquad
0<C_{\rm CFL}\le 0.44\ (k=4),
\\&
0<C_{\rm CFL}\le 0.58\ (k=5),
\qquad
0<C_{\rm CFL}\le 0.56\ (k=6),
\end{aligned}
\]
where $C_{\rm CFL}=a\Delta t/\Delta x$.
\end{proposition}
\begin{proof}
The SSP-RK3 amplification matrix is the standard third-order stability polynomial applied to the semi-discrete symbol:
\begin{equation}\label{eq:Gmatrix2}
\bm G
=\bm I+\bm Z+\frac12\bm Z^2+\frac16\bm Z^3,
\qquad
\bm Z:=\frac{a\Delta t}{\Delta x}\widehat{\bm G},
\end{equation}
where $\widehat{\bm G}=\Delta x\widetilde{\bm G}$. For this sampled study, we estimate admissible CFL numbers by requiring $\rho(\bm G)\le1$ for the sampled values of $\xi$. The resulting sampled admissible CFL estimates are stated above and are not claimed to be exact optimized CFL bounds.
\end{proof}

The sampled spectral-radius curves used to obtain these CFL estimates are reported in the supplementary material.

\begin{remark}
The analysis in this section is restricted to the linear-advection model problem and to the underlying schemes without OE. It explains the accuracy and sampled CFL estimates for the linear backbone of FAM, but it is not a nonlinear stability theory for the fully corrected schemes or for general systems.
\end{remark}

\section{Numerical results}\label{sec:NumTest}

This section tests whether the fixed $\mathbb P^1$ moment representation retains the designed smooth order, how the end-of-step OE correction behaves on representative discontinuous problems, and how runtime is distributed in Euler comparisons. All computations use uniform Cartesian meshes and SSP-RK3; the test set includes Burgers' equation, a nonconvex scalar law, and the compressible Euler equations with $\gamma=1.4$ unless otherwise stated.

\paragraph{Time stepping} For smooth problems we choose $\Delta t=C_{\rm CFL}\Delta x^{k/3}/\alpha_x$ in 1D and $\Delta t=C_{\rm CFL}/(\alpha_x/\Delta x^{k/3}+\alpha_y/\Delta y^{k/3})$ in 2D, and we checked that further reducing $\Delta t$ on the listed meshes does not change the observed orders. Here $\alpha_x$ and $\alpha_y$ denote the maximum Rusanov wave-speed estimates over the mesh in the corresponding coordinate directions. For discontinuous problems we use $\Delta t=C_{\rm CFL}\Delta x/\alpha_x$ in 1D and $\Delta t=C_{\rm CFL}/(\alpha_x/\Delta x+\alpha_y/\Delta y)$ in 2D. Using Proposition~\ref{prop:CFL} as a sampled linear guide, the FAM discontinuous tests use $C_{\rm CFL}=0.4$ for FAM3--FAM4 and $0.55$ for FAM5--FAM6; the comparison methods use the values reported in their original papers ($0.6$ for MR-WENO3/MR-WENO5/HWENO-U and $0.45$ for OE-HWENO). Smooth tests use $C_{\rm CFL}=0.1$ for all methods. 

\paragraph{Comparison framework} The comparison set consists of FAM$k$ ($k=3,4,5,6$), MR-WENO3, MR-WENO5 \cite{ZS_MRWENO}, HWENO-U \cite{FQZ}, and OE-HWENO \cite{FanWu}. All simulations use double-precision Fortran~95 on an Intel(R) Xeon(R) CPU E5-2640 v4 @~2.40~GHz. FAM stores two moments per conserved component per cell in 1D and three in 2D. Reported wall-clock times are complete-run times within the same code base, including reconstruction, flux evaluation, OE work, and optional BP/PP limiting. The HWENO-type Euler baselines follow the local characteristic-variable implementations. The emphasized matched-order comparisons are FAM3 versus MR-WENO3, FAM5 versus MR-WENO5 and HWENO-U, and FAM6 versus OE-HWENO.

The discussion proceeds from smooth accuracy and selected matched-order cost comparisons to discontinuous-problem behavior, OE ablation, and runtime attribution. Table~\ref{tab:test-config} summarizes the benchmark-specific OE choice, BP/PP setting, and main role of each test; additional sensitivity and ablation figures are in the supplementary material.

\begin{table}[t]
	\centering
	\caption{Configuration summary for the benchmarks reported in the main paper.}
	\label{tab:test-config}
	\scriptsize
	\setlength{\tabcolsep}{4pt}
	\renewcommand{\arraystretch}{1.08}
	\begin{tabular}{@{}p{0.75cm}p{2.15cm}p{2.55cm}p{0.9cm}p{0.95cm}p{3.45cm}@{}}
		\toprule[1.2pt]
		Ex. & Problem & Role & OE & BP/PP & Main features \\
		\midrule
		4.1 & smooth Burgers & 1D scalar accuracy / cost & II & off & designed order and matched-order cost \\
		4.2 & smooth Euler (1D/2D) & system accuracy / cost & II & off & smooth-system accuracy with componentwise variables \\
		4.3 & Buckley--Leverett & nonconvex scalar benchmark & II & on & compound-wave resolution \\
		4.4 & Shu--Osher & 1D oscillatory shock interaction & II & off & shock-wave resolution; main OE-ablation benchmark \\
		4.5 & Lax shock tube & moderate Euler Riemann problem & II & off & baseline shock-capturing behavior \\
		4.6 & multiscale Lax & coefficient comparison & I / II & on & Type~II is preferable on multiscale flows \\
		4.7 & Sedov / LeBlanc & extreme 1D Euler tests & II & on & strong-shock behavior and complete-run wall-clock cost \\
		4.8 & double Mach reflection & 2D strong-shock benchmark & II & on & 2D shock-capturing behavior and runtime attribution \\
		\bottomrule[1.2pt]
	\end{tabular}
\end{table}

\subsection{Smooth accuracy and matched-order cost}

\begin{exmp}
\label{ex1:Burgers}
\textbf{Burgers equation.} We first consider the one-dimensional Burgers equation with smooth periodic data $u(x,0)=0.5+\sin(\pi x)$ on $[0,2]$. The final time is $T=0.5/\pi$, so the exact solution remains smooth. Table~\ref{tab:1dBur} reports the $L^1$ errors of the cell averages, while the supplementary material records the corresponding errors of the scaled first-order moments. Figure~\ref{fig:1dBur} plots the error against wall-clock time.

We see that the FAM$k$ schemes attain the expected orders: cell averages display a $(k+1)$th-order pattern consistent with the semi-discrete linear-backbone superconvergence mechanism in Proposition~\ref{prop:superavg}; the time-step refinement checks indicate that the fully discrete errors are not dominated by RK temporal error. While evolved scaled first-order moments converge with order $k$ (see the supplementary table). In the reported matched-order error-time curves, FAM5 gives smaller errors than MR-WENO5 and HWENO-U at comparable runtime in the plotted range, while FAM6 has comparable error-time behavior to OE-HWENO. For the cell-average superconvergence tables we also repeated the runs with smaller $\Delta t$, and observed unchanged orders.

\begin{figure}[t]
\centering
\includegraphics[width=0.86\textwidth]{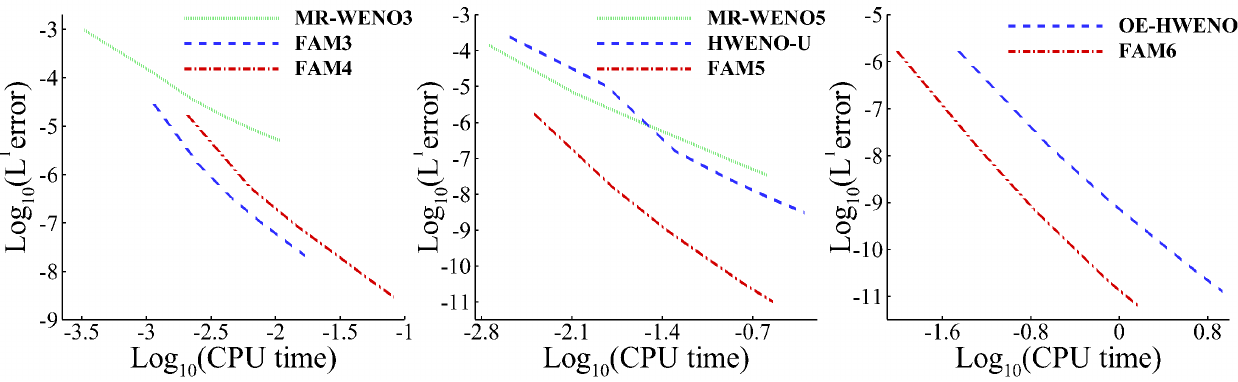}
\caption{$L^1$ error versus wall-clock time for the smooth Burgers test in Example~\ref{ex1:Burgers}.}
\label{fig:1dBur}
\end{figure}

\begin{table}[t]
\centering
\caption{$L^1$ errors of the cell averages for the smooth Burgers test in Example~\ref{ex1:Burgers}.}
\label{tab:1dBur}
\scriptsize
\resizebox{\textwidth}{!}{
\begin{tabular}{c|cccccccc}
\toprule[1.2pt]
$N_x$ & \multicolumn{2}{c}{FAM3} & \multicolumn{2}{c}{FAM4} & \multicolumn{2}{c}{FAM5} & \multicolumn{2}{c}{FAM6} \\
\cmidrule(l){2-3}\cmidrule(l){4-5}\cmidrule(l){6-7}\cmidrule(l){8-9}
& $L^1$ error & order & $L^1$ error & order & $L^1$ error & order & $L^1$ error & order \\
\midrule
30 & 2.10E-05 & --   & 1.67E-05 & --   & 1.73E-06 & --   & 1.64E-06 & --   \\
60 & 1.45E-06 & 3.86 & 5.28E-07 & 4.98 & 1.65E-08 & 6.71 & 1.40E-08 & 6.88 \\
90 & 2.95E-07 & 3.93 & 8.09E-08 & 4.63 & 9.88E-10 & 6.94 & 7.29E-10 & 7.28 \\
120& 9.68E-08 & 3.87 & 2.11E-08 & 4.68 & 1.41E-10 & 6.77 & 9.88E-11 & 6.95 \\
150& 4.07E-08 & 3.89 & 7.09E-09 & 4.88 & 3.17E-11 & 6.68 & 2.08E-11 & 6.99 \\
180& 1.99E-08 & 3.92 & 2.98E-09 & 4.76 & 9.80E-12 & 6.44 & 6.04E-12 & 6.78 \\
\midrule
$N_x$ & \multicolumn{2}{c}{MR-WENO3} & \multicolumn{2}{c}{MR-WENO5} & \multicolumn{2}{c}{HWENO-U} & \multicolumn{2}{c}{OE-HWENO} \\
\cmidrule(l){2-3}\cmidrule(l){4-5}\cmidrule(l){6-7}\cmidrule(l){8-9}
& $L^1$ error & order & $L^1$ error & order & $L^1$ error & order & $L^1$ error & order \\
\midrule
30 & 9.10E-04 & --   & 1.41E-04 & --   & 1.38E-04 & --   & 1.67E-06 & --   \\
60 & 1.22E-04 & 2.89 & 6.66E-06 & 4.40 & 2.66E-06 & 5.70 & 1.47E-08 & 6.83 \\
90 & 3.81E-05 & 2.88 & 1.01E-06 & 4.65 & 1.63E-07 & 6.88 & 1.04E-09 & 6.53 \\
120& 1.63E-05 & 2.95 & 2.51E-07 & 4.84 & 2.81E-08 & 6.11 & 1.65E-10 & 6.40 \\
150& 8.37E-06 & 2.98 & 8.40E-08 & 4.91 & 7.96E-09 & 5.65 & 3.97E-11 & 6.37 \\
180& 4.85E-06 & 3.00 & 3.40E-08 & 4.97 & 3.10E-09 & 5.17 & 1.21E-11 & 6.54 \\
\bottomrule[1.2pt]
\end{tabular}}
\end{table}
\end{exmp}

\begin{exmp}
\label{ex:Order_Euler}
\textbf{Euler equations.}
We next test the schemes on smooth Euler solutions. In one dimension we use $(\rho_0,\mu_0,p_0)=(1+0.2\sin(\pi x),1,1)$ on $[0,2]$ with periodic boundaries. In two dimensions we use $(\rho_0,\mu_0,\nu_0,p_0)=(1+0.2\sin(\pi(x+y)),1,1,1)$ on $[0,2]^2$, again with periodic boundaries. All runs stop at $T=2$, so the exact solution is a simple translation of the initial wave.

Tables~\ref{tab:1dEuler} and~\ref{tab:2dEuler} report $L^1$ errors of the density cell averages. The FAM$k$ schemes achieve the expected convergence: density cell averages show an empirical $(k+1)$th-order pattern, while the supplementary material records the corresponding one-dimensional moment-error tables; the two-dimensional moment data show the same behavior and are omitted for brevity. Figure~\ref{fig:1dEuler} compares 1D density errors against runtime. 

Since FAM5/HWENO-U and FAM6/OE-HWENO are same-$\mathbb P^1$-state comparisons at the corresponding accuracy levels, they mainly probe the stagewise reconstruction design rather than the number of evolved local DOFs.

\begin{table}[t]
\centering
\caption{$L^1$ errors of density cell averages for the 1D smooth Euler test in Example~\ref{ex:Order_Euler}.}
\label{tab:1dEuler}
\scriptsize
\resizebox{\textwidth}{!}{
\begin{tabular}{c|cccccccc}
\toprule[1.2pt]
$N_x$ & \multicolumn{2}{c}{FAM3} & \multicolumn{2}{c}{FAM4} & \multicolumn{2}{c}{FAM5} & \multicolumn{2}{c}{FAM6} \\
\cmidrule(l){2-3}\cmidrule(l){4-5}\cmidrule(l){6-7}\cmidrule(l){8-9}
& $L^1$ error & order & $L^1$ error & order & $L^1$ error & order & $L^1$ error & order \\
\midrule
20 & 1.18E-05 & --   & 6.86E-06 & --   & 1.68E-08 & --   & 9.84E-09 & --   \\
30 & 2.22E-06 & 4.12 & 9.62E-07 & 4.84 & 1.13E-09 & 6.65 & 5.93E-10 & 6.93 \\
40 & 6.88E-07 & 4.07 & 2.34E-07 & 4.92 & 1.78E-10 & 6.43 & 8.04E-11 & 6.95 \\
50 & 2.80E-07 & 4.03 & 7.80E-08 & 4.92 & 4.34E-11 & 6.33 & 1.71E-11 & 6.94 \\
60 & 1.34E-07 & 4.02 & 3.18E-08 & 4.92 & 1.40E-11 & 6.21 & 4.80E-12 & 6.97 \\
70 & 7.23E-08 & 4.02 & 1.48E-08 & 4.94 & 5.42E-12 & 6.15 & 1.66E-12 & 6.90 \\
\midrule
$N_x$ & \multicolumn{2}{c}{MR-WENO3} & \multicolumn{2}{c}{MR-WENO5} & \multicolumn{2}{c}{HWENO-U} & \multicolumn{2}{c}{OE-HWENO} \\
\cmidrule(l){2-3}\cmidrule(l){4-5}\cmidrule(l){6-7}\cmidrule(l){8-9}
& $L^1$ error & order & $L^1$ error & order & $L^1$ error & order & $L^1$ error & order \\
\midrule
20 & 4.63E-03 & --   & 9.22E-05 & --   & 8.07E-06 & --   & 1.80E-07 & --   \\
30 & 1.40E-03 & 2.94 & 1.23E-05 & 4.96 & 1.05E-06 & 5.00 & 1.19E-08 & 6.70 \\
40 & 5.96E-04 & 2.98 & 2.94E-06 & 4.98 & 2.50E-07 & 5.00 & 1.66E-09 & 6.83 \\
50 & 3.06E-04 & 2.99 & 9.67E-07 & 4.99 & 8.18E-08 & 5.00 & 3.57E-10 & 6.89 \\
60 & 1.77E-04 & 2.99 & 3.89E-07 & 4.99 & 3.29E-08 & 5.00 & 1.01E-10 & 6.91 \\
70 & 1.12E-04 & 2.99 & 1.80E-07 & 5.00 & 1.52E-08 & 5.00 & 3.48E-11 & 6.94 \\
\bottomrule[1.2pt]
\end{tabular}}
\end{table}

\begin{figure}[t]
\centering
\includegraphics[width=0.86\textwidth]{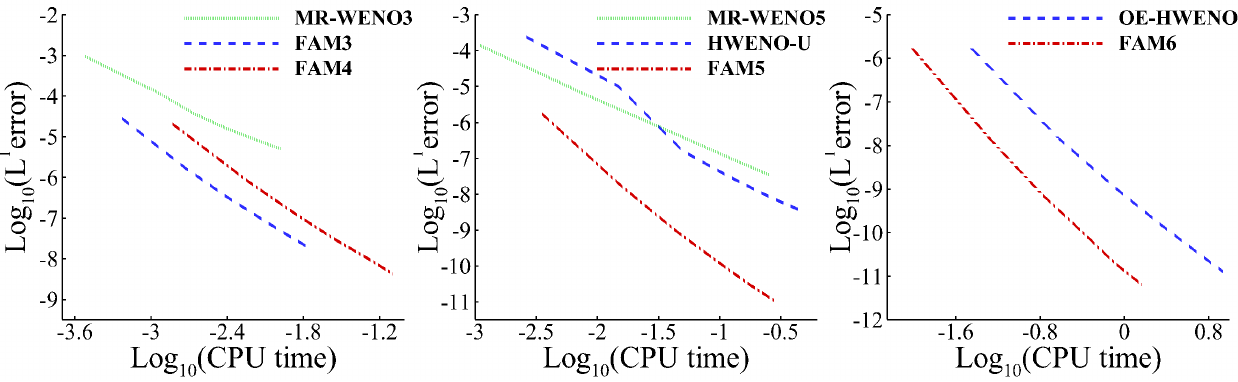}
\caption{$L^1$ error of the density cell averages versus wall-clock time for the 1D smooth Euler test in Example~\ref{ex:Order_Euler}.}
\label{fig:1dEuler}
\end{figure}

\begin{table}[t]
\centering
\caption{$L^1$ errors of density cell averages for the 2D smooth Euler test in Example~\ref{ex:Order_Euler}.}
\label{tab:2dEuler}
\scriptsize
\resizebox{\textwidth}{!}{
\begin{tabular}{c|cccccccc}
\toprule[1.2pt]
$N_x$ & \multicolumn{2}{c}{FAM3} & \multicolumn{2}{c}{FAM4} & \multicolumn{2}{c}{FAM5} & \multicolumn{2}{c}{FAM6} \\
\cmidrule(l){2-3}\cmidrule(l){4-5}\cmidrule(l){6-7}\cmidrule(l){8-9}
& $L^1$ error & order & $L^1$ error & order & $L^1$ error & order & $L^1$ error & order \\
\midrule
20$\times$20  & 3.14E-04 & --   & 1.28E-04 & --   & 8.16E-07 & --   & 9.46E-07 & --   \\
40$\times$40  & 1.64E-05 & 4.26 & 3.85E-06 & 5.05 & 8.13E-09 & 6.65 & 7.36E-09 & 7.01 \\
60$\times$60  & 3.01E-06 & 4.19 & 5.11E-07 & 4.98 & 5.56E-10 & 6.61 & 4.24E-10 & 7.04 \\
80$\times$80  & 9.14E-07 & 4.14 & 1.23E-07 & 4.96 & 8.44E-11 & 6.56 & 5.59E-11 & 7.04 \\
100$\times$100& 3.65E-07 & 4.11 & 4.05E-08 & 4.97 & 1.98E-11 & 6.51 & 1.18E-11 & 6.96 \\
120$\times$120& 1.73E-07 & 4.09 & 1.63E-08 & 4.98 & 6.11E-12 & 6.44 & 3.31E-12 & 6.99 \\
\midrule
$N_x$ & \multicolumn{2}{c}{MR-WENO3} & \multicolumn{2}{c}{MR-WENO5} & \multicolumn{2}{c}{HWENO-U} & \multicolumn{2}{c}{OE-HWENO} \\
\cmidrule(l){2-3}\cmidrule(l){4-5}\cmidrule(l){6-7}\cmidrule(l){8-9}
& $L^1$ error & order & $L^1$ error & order & $L^1$ error & order & $L^1$ error & order \\
\midrule
20$\times$20  & 2.50E-02 & --   & 3.60E-03 & --   & 9.62E-04 & --   & 9.96E-06 & --   \\
40$\times$40  & 3.56E-03 & 2.81 & 1.21E-04 & 4.90 & 3.30E-06 & 8.19 & 1.09E-07 & 6.52 \\
60$\times$60  & 1.07E-03 & 2.96 & 1.62E-05 & 4.96 & 1.39E-07 & 7.81 & 6.72E-09 & 6.87 \\
80$\times$80  & 4.55E-04 & 2.98 & 3.87E-06 & 4.97 & 1.96E-08 & 6.81 & 9.30E-10 & 6.87 \\
100$\times$100& 2.33E-04 & 2.99 & 1.28E-06 & 4.97 & 5.42E-09 & 5.77 & 1.99E-10 & 6.90 \\
120$\times$120& 1.35E-04 & 2.99 & 5.18E-07 & 4.96 & 2.07E-09 & 5.28 & 5.65E-11 & 6.92 \\
\bottomrule[1.2pt]
\end{tabular}}
\end{table}
\end{exmp}

\subsection{Representative discontinuous problems}
We first report the Buckley--Leverett benchmark with the BP/PP setting stated in Table~\ref{tab:test-config}, followed by Shu--Osher and Lax without auxiliary BP/PP limiting; the multiscale and strong-shock problems then activate the standard BP/PP limiter. Unless stated otherwise, Euler examples use componentwise conservative-variable reconstruction; the later variable-space comparison changes this choice only within FAM5.

\begin{exmp}
\label{ex:1dBL}
\textbf{Buckley--Leverett problem.} We examine the nonlinear, nonconvex Buckley--Leverett equation $\partial_t u+\partial_x f(u)=0$ with $f(u)=\frac{4u^2}{4u^2+(1-u)^2}$ and initial condition $u(x,0)=1$ for $x\in[-0.5,0]$ and $u(x,0)=0$ elsewhere. The computation is performed on $[-1,1]$ with periodic boundary conditions up to the final time $T=0.4$. This test features an entropy solution containing both shocks and rarefactions. Figure~\ref{fig:1dBL} shows that the FAM$k$ solutions remain close to the exact entropy solution near the compound wave and are comparable to the well-resolved baseline profiles.

\begin{figure}[t]
\centering
\includegraphics[width=0.86\textwidth]{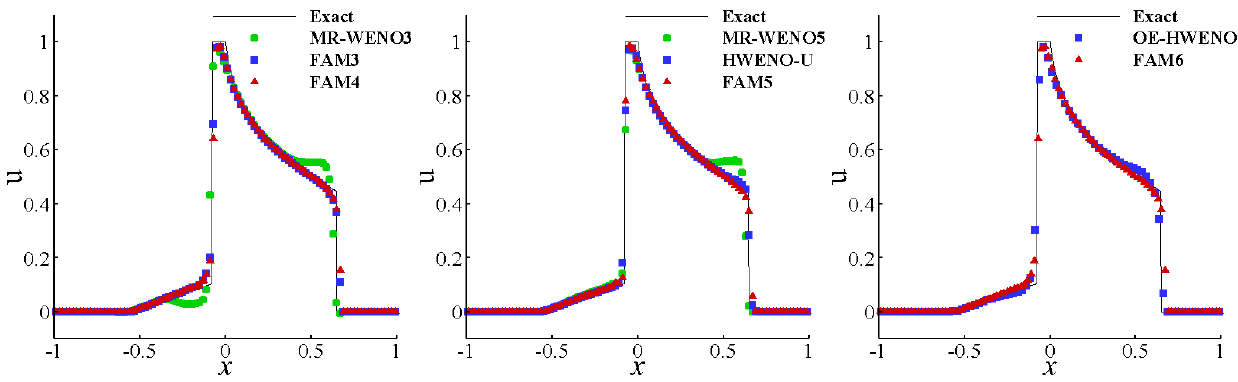}
\caption{Cell averages for the Buckley--Leverett problem in Example~\ref{ex:1dBL} on 100 cells. The black solid curve is the exact entropy solution.}
\label{fig:1dBL}
\end{figure}
\end{exmp}

\begin{exmp}\label{ex:1dShuOsher}
\textbf{Shu--Osher shock--density-wave interaction.} We next consider the interaction of a right-moving Mach~3 shock with sinusoidal density waves. The initial condition is $(\rho_0,\mu_0,p_0)=(3.857143,2.629369,10.333333)$ for $-5\le x<-4$ and $(\rho_0,\mu_0,p_0)=(1+0.2\sin(5x),0,1)$ for $-4<x\le5$. The left boundary is prescribed by the pre-shock state, and the right boundary uses an outflow condition. The final time is $T=1.8$. No auxiliary BP/PP limiter is used in this example.

Figure~\ref{fig:1dShuOsher} shows that the FAM schemes capture the main shock and the post-shock oscillatory structures without visible nonphysical oscillations. The fine-grid reference solution is computed by FAM5 on 60{,}000 cells.

\begin{figure}[t]
	\centering
	\includegraphics[width=0.86\textwidth]{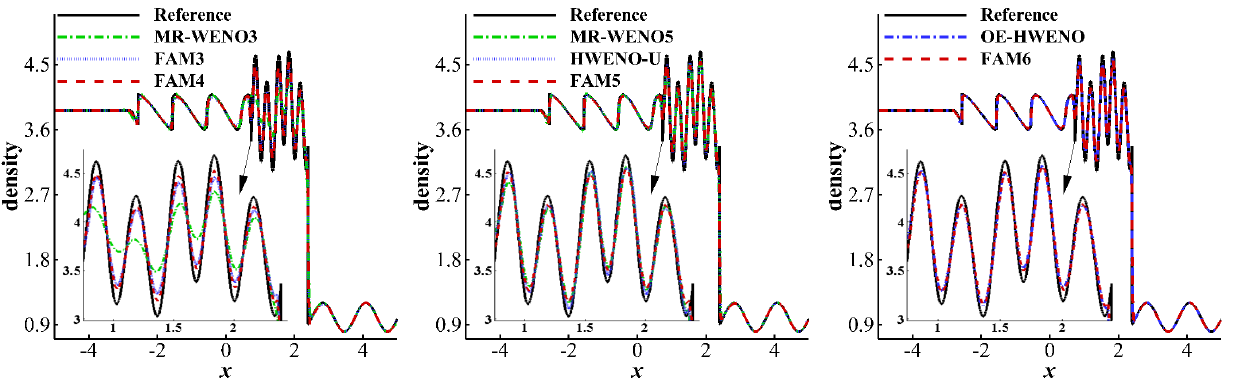}
	\caption{Density for the Shu--Osher problem in Example~\ref{ex:1dShuOsher} on 400 cells.}
	\label{fig:1dShuOsher}
\end{figure}
\end{exmp}

\begin{exmp}\label{ex:1dLax}
\textbf{Lax shock tube.} This classical Riemann problem uses the discontinuous initial data $(\rho_0,\mu_0,p_0)=(0.445,0.698,3.528)$ for $-0.5\le x<0$ and $(0.5,0,0.571)$ for $0<x\le0.5$ with outflow boundary conditions. The final time is $T=0.16$. No auxiliary BP/PP limiter is used in this example. Figure~\ref{fig:1dLax} shows that the FAM$k$ schemes resolve the elementary-wave pattern cleanly and remain comparable with the baseline methods on this standard test.

\begin{figure}[t]
\centering
\includegraphics[width=0.76\textwidth]{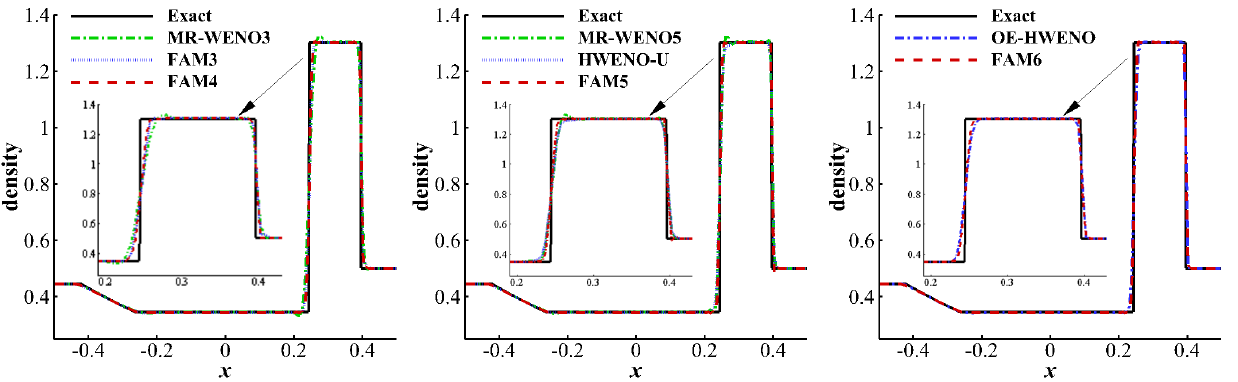}
\caption{Density for the Lax shock-tube problem in Example~\ref{ex:1dLax} on 200 cells.}
\label{fig:1dLax}
\end{figure}
\end{exmp}

\begin{exmp}\label{ex:1dMultiScale}
\textbf{Multiscale Lax problem.} This benchmark, adapted from \cite{CHLW}, is designed to differentiate the two OE coefficients. The initial data combine the standard Lax shock tube with copies scaled by factors $100$ and $0.01$: $(\rho_0,\mu_0,p_0)=(44.5,0.698,352.8)$ for $-15\le x<-10$, $(50,0,57.1)$ for $-10<x\le0$, $(0.00445,0.698,\allowbreak 0.03528)$ for $0\le x<10$, and $(0.005,0,0.00571)$ for $10<x\le15$ with outflow boundary conditions. The final time is $T=1.4$. Here ``FAM$k$-I'' and ``FAM$k$-II'' denote the schemes using the Type~I and Type~II damping coefficients, respectively.

The reference solution plotted in Figure~\ref{fig:1dMultiScale} is computed with the FAM5-II scheme on 64{,}000 cells.
Within that qualitative comparison, FAM3-I and FAM5-I develop visible undershoots in the inset, whereas FAM3-II and FAM5-II substantially reduce them and retain better-resolved small structures. Accordingly, the strongly multiscale Euler tests below use the Type~II coefficient.

\begin{figure}[!htbp]
\centering
\includegraphics[width=0.72\textwidth]{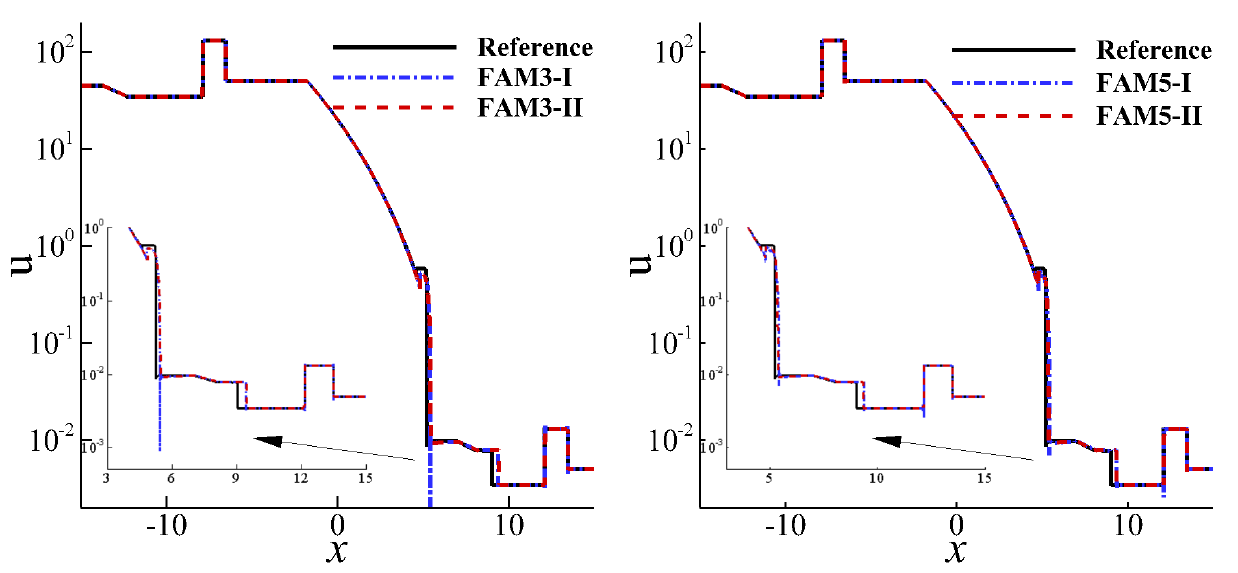}
\caption{Density for the multiscale Lax shock-tube test in Example~\ref{ex:1dMultiScale} on 6400 cells. }
\label{fig:1dMultiScale}
\end{figure}
\end{exmp}

\begin{exmp}\label{ex:1dExtreme}
\textbf{Extreme one-dimensional Euler tests.} We next test the schemes on two demanding one-dimensional Euler problems. The first is a Sedov point-blast problem on $[-2,2]$ with background density $\rho_0=1$, velocity $\mu_0=0$, background total energy $E_0=10^{-12}$, and a single central cell carrying energy $3.2\times10^6/\Delta x$. The final time is $T=10^{-3}$. The second is the LeBlanc-type shock tube on $[-10,10]$ with initial data $(\rho_0,\mu_0,p_0)=(2,0,10^9)$ for $x<0$ and $(10^{-3},0,1)$ for $x>0$, evolved to $T=10^{-4}$ under the adiabatic-index convention stated at the beginning of this section. Outflow boundary conditions are used.

These tests require the standard PP limiter for all methods in the comparison set. Figure~\ref{fig:1dExtreme} shows FAM density profiles on these severe tests, and Table~\ref{tab:1dCPU} gives complete-run wall-clock times. One can see that FAM5 has lower reported times than MR-WENO5 and HWENO-U, and FAM6 has lower reported times than OE-HWENO; the smallest reported ratio is $0.1$ for FAM6/OE-HWENO on LeBlanc under the stated implementation and CFL conventions.

\begin{figure}[!htbp]
\centering
\includegraphics[width=0.86\textwidth]{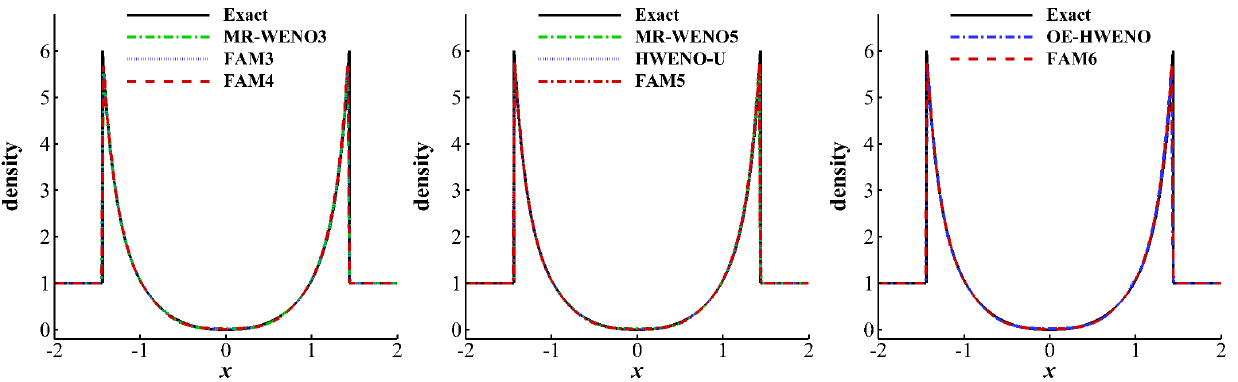}\\[1ex]
\includegraphics[width=0.86\textwidth]{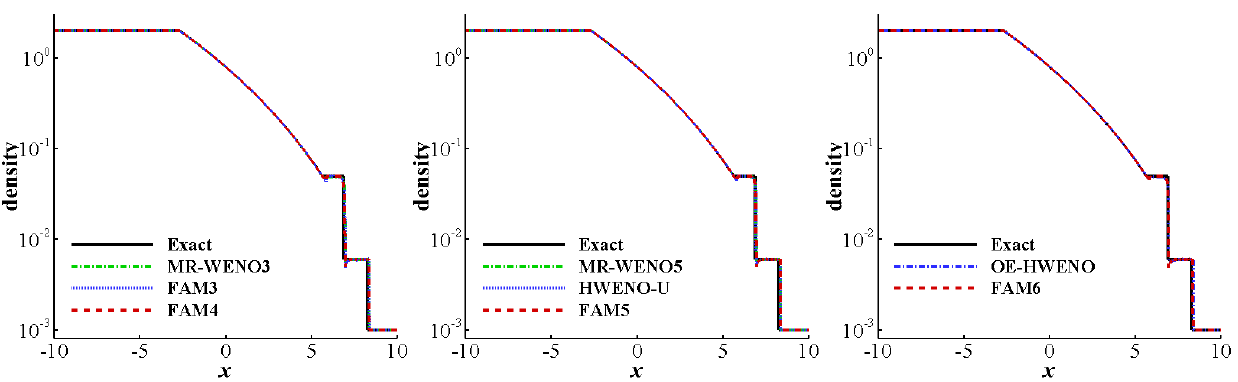}
\caption{Density for the Sedov and LeBlanc problems in Example~\ref{ex:1dExtreme}.}
\label{fig:1dExtreme}
\end{figure}

\begin{table}[!htbp]
\centering
\begin{threeparttable}
\caption{Wall-clock times and ratios for the complete Sedov/LeBlanc runs in Example~\ref{ex:1dExtreme}, including PP limiting. In each time pair, the FAM time is listed first and the baseline time second, both in seconds. All methods use SSP-RK3, so the ratios reflect both per-stage cost and the different numbers of time steps induced by the chosen CFL numbers.}
\label{tab:1dCPU}
\small
\begin{tabular}{lcccc}
\toprule[1.2pt]
Pair & Sedov: time pair (s) & Ratio & LeBlanc: time pair (s) & Ratio \\
\midrule
FAM3 vs MR-WENO3 & $1.48$ vs $1.04$ & $1.42$ & $37.1$ vs $23.0$ & $1.61$ \\
FAM5 vs MR-WENO5 & $0.99$ vs $1.20$ & $0.83$ & $20.1$ vs $31.3$ & $0.64$ \\
FAM5 vs HWENO-U  & $0.99$ vs $1.62$ & $0.61$ & $20.1$ vs $36.4$ & $0.55$ \\
FAM6 vs OE-HWENO & $1.16$ vs $4.53$ & $0.26$ & $19.9$ vs $196.8$ & $0.10$ \\
\bottomrule[1.2pt]
\end{tabular}
\end{threeparttable}
\end{table}
\end{exmp}

\begin{exmp}\label{ex:DoubleMach}
\textbf{Double Mach reflection.} We simulate the classical double Mach reflection problem \cite{WC} on $[0,3]\times[0,1]$. The states ahead of and behind the incident shock are prescribed as $(\rho,\mu_0,\nu_0,p)=(8,8.25\cos(\frac{\pi}{6}),-8.25\sin(\frac{\pi}{6}),116.5)$ behind the shock and $(1.4,0,0,1)$ ahead of the shock. A Mach~10 shock initially intersects the bottom boundary at $(x,y)=(1/6,0)$ and makes a $60^\circ$ angle with the $x$-axis. The left and right boundaries are inflow and outflow, respectively; the top boundary follows the exact shock motion; on the bottom boundary we impose the exact post-shock state on $[0,1/6]\times\{0\}$ and a reflective condition elsewhere. 
Figure~\ref{fig:DoubleMach} shows density contours at $T=0.2$ on a $1440\times480$ mesh. With the PP limiter activated, the FAM solutions capture the main shock structure, slip lines, and double-Mach stem without visible nonphysical oscillations. In the reported same $\mathbb P^1$-state timings, FAM5 has lower reported times than HWENO-U and OE-HWENO; Subsection~\ref{subsec:ablation} gives the runtime attribution.

\begin{figure}[t]
\centering
\includegraphics[width=0.97\textwidth]{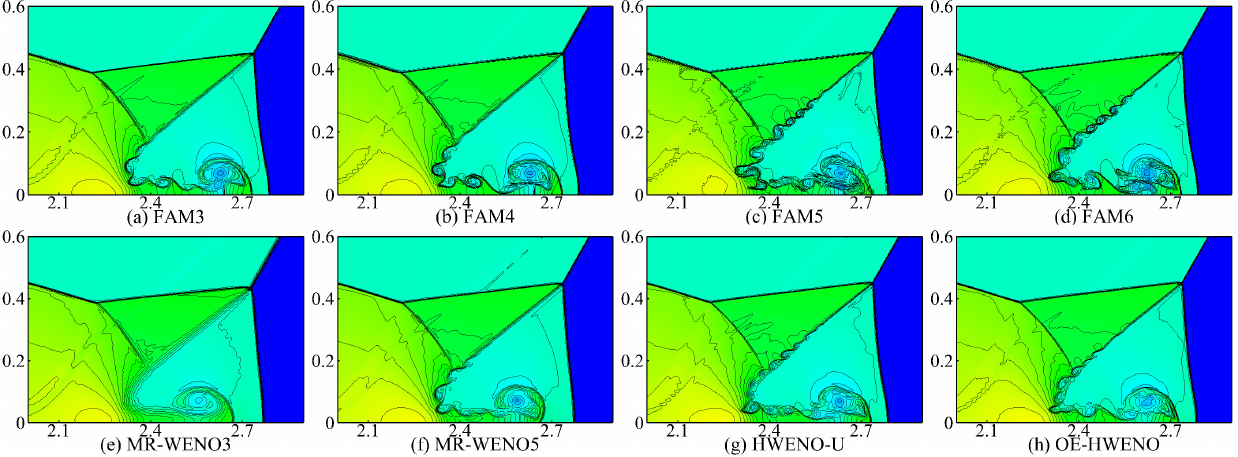}
\caption{30 equally spaced contours of density from  1.5 to 22.0 for the double-Mach-reflection test in Example~\ref{ex:DoubleMach} at $T=0.2$ on a $1440\times480$ mesh. }
\label{fig:DoubleMach}
\end{figure}
\end{exmp}

\subsection{\texorpdfstring{OE ablation and runtime attribution}{OE ablation and runtime attribution}}\label{subsec:ablation}

This subsection examines the end-of-step OE correction and the main runtime contributors on the Shu--Osher and double-Mach benchmarks.

Figure~\ref{fig:1dShuOsherOE} gives the main OE ablation. Without OE, even with the standard PP limiter activated to maintain admissibility in this ablation run, FAM5 exhibits visible nonphysical oscillations; the OE choices substantially reduce them in this plot. The supplementary double-Mach ablation gives the same qualitative conclusion, with Type~II giving slightly better-resolved structures on the more multiscale flow.

\begin{figure}[t]
\centering
\includegraphics[width=0.97\textwidth]{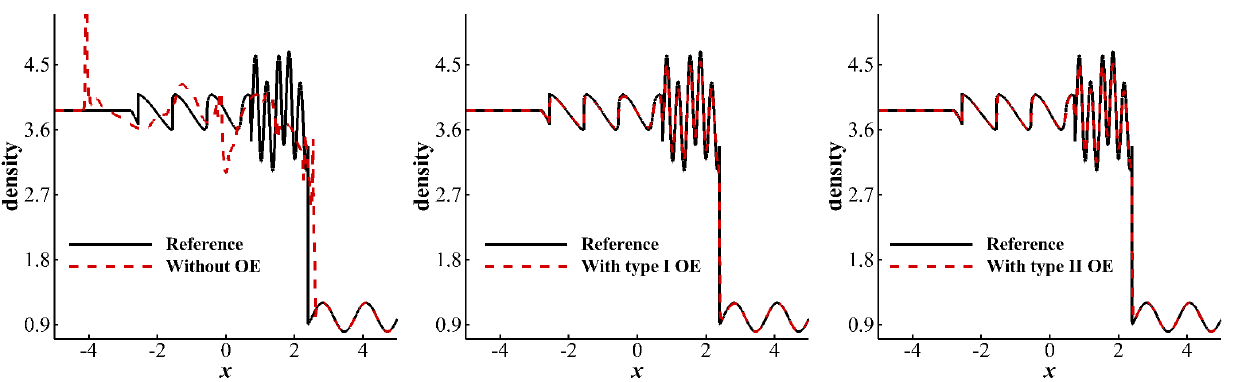}
\caption{FAM5 density for the Shu--Osher problem without OE and with the Type~I/Type~II OE processes on 400 cells. The no-OE computation uses the standard PP limiter to maintain admissibility in this no-OE ablation run.}
\label{fig:1dShuOsherOE}
\end{figure}

Table~\ref{tab:runtime-main} reports the runtime attribution.  In Panel~A, reconstruction accounts for about $83\%$ and $81\%$ of the HWENO-U and OE-HWENO runtimes on Shu--Osher, and about $85\%$ and $77\%$ on double Mach; so the dominant measured component in these compact same $\mathbb P^1$-state HWENO baselines is stagewise reconstruction. For FAM5, the end-of-step OE correction remains secondary: the total runtime is about $0.67$ of HWENO-U and $0.20$ of OE-HWENO on Shu--Osher, and about $0.42$ and $0.14$ on double Mach. Panel~B shows the extra cost of switching FAM5 itself from conservative to local characteristic variables. Because the flux, RK integrator, limiter logic, and timing convention are shared, Table~\ref{tab:runtime-main} is a runtime-attribution study within one implementation.

\begin{table}[t]
\centering
\begin{threeparttable}
\caption{Runtime attribution for representative examples within the shared implementation and hardware setting described at the start of this section. Panel~A reports the same $\mathbb P^1$-state compact FAM/HWENO comparison. Panel~B compares componentwise conservative-variable and local characteristic-variable FAM5 implementations. All times are in seconds.}
\label{tab:runtime-main}
\scriptsize
\begin{tabular}{llcccc}
\toprule[1.2pt]
Problem & Variant & Total & Reconstruction & OE & Other \\
\midrule
\multicolumn{6}{l}{\textit{Panel A: same $\mathbb P^1$-state compact FAM/HWENO comparison}} \\
Shu--Osher
& FAM5     & 0.135    & 0.058    & 0.023    & 0.054 \\
& HWENO-U  & 0.201    & 0.167    & --       & 0.034 \\
& OE-HWENO & 0.681    & 0.554    & 0.079    & 0.048 \\
\cmidrule(lr){1-6}
Double Mach
& FAM5     & 2.23E+04 & 1.04E+04 & 4.24E+03 & 7.66E+03 \\
& HWENO-U  & 5.29E+04 & 4.49E+04 & --       & 8.00E+03 \\
& OE-HWENO & 1.56E+05 & 1.20E+05 & 2.59E+04 & 1.01E+04 \\
\midrule
\multicolumn{6}{l}{\textit{Panel B: FAM5 conservative vs local characteristic variables}} \\
Shu--Osher & conservative & 0.135 & 0.058 & 0.023 & 0.054 \\
 & characteristic & 0.169 & 0.095 & 0.026 & 0.048 \\
\cmidrule(lr){1-6}
Double Mach & conservative & 2.23E+04 & 1.04E+04 & 4.24E+03 & 7.66E+03 \\
 & characteristic & 4.02E+04 & 2.87E+04 & 3.98E+03 & 7.52E+03 \\
\bottomrule[1.2pt]
\end{tabular}
\end{threeparttable}
\end{table}

For FAM5, OE accounts for about $17\%$ of total runtime on Shu--Osher and $19\%$ on double Mach, whereas reconstruction accounts for about $43\%$ and $47\%$. In Panel~B, switching from conservative to local characteristic variables multiplies reconstruction time by about $1.64$ and $2.76$, increasing total FAM5 runtime by factors of about $1.25$ and $1.80$; the OE and ``other'' buckets change only mildly. Thus the lower reported timings in these selected comparisons come from compact linear RK-stage reconstruction plus single end-of-step OE, with an additional lower measured reconstruction cost in these tests when FAM5 uses conservative rather than local characteristic variables.

At the plotted resolution, conservative- and characteristic-variable FAM5 solutions are essentially indistinguishable in the displayed density fields on both benchmarks; supplementary figures show the same behavior. At these resolutions, local characteristic-variable reconstruction within FAM5 adds measured cost without an observable gain in those plotted density fields.

\section{Conclusion}\label{sec:conclusion}

We have developed a compact finite-average-moment (FAM) framework that shifts the main computational burden in high-order shock-capturing away from repeated nonlinear reconstruction and toward compact linear reconstruction combined with a single end-of-step oscillation-elimination correction. The evolved state is kept fixed at the $\mathbb{P}^1$ moment level, consisting only of cell averages and scaled first-order moments, while third- through sixth-order accuracy is recovered by compact linear moment reconstructions. In this way, the target order is decoupled from the number of locally evolved degrees of freedom.

 The core algorithmic contribution is the combined design of the method: a minimal fixed-moment state, polynomially exact compact linear RK-stage reconstructions, and a single derivative-free oscillation-elimination (OE) correction applied only at the end of each time step. The OE step preserves cell averages exactly and damps only the scaled first-order moments through an explicit exponential update. Compared with traditional WENO/HWENO-type schemes, the proposed framework avoids repeated stagewise nonlinear reconstructions. Compared with DG-based OEDG/COS(DG) frameworks, it keeps the evolved state fixed at the $\mathbb{P}^1$ level rather than updating higher-degree polynomial states.
 
 The analysis and numerical experiments support this design. For the linear-advection backbone without OE, Fourier analysis identifies a $k$th-order physical mode, a strongly damped spurious mode, $(k+1)$th-order cell-average superconvergence, and sampled admissible CFL estimates. Reconstruction exactness and consistency results further justify the compact fixed-moment formulation. Numerically, the FAM schemes achieve the expected smooth accuracy and produce robust nonoscillatory profiles across representative scalar and compressible Euler benchmarks, including shock interactions, multiscale tests, strong-shock problems, and the two-dimensional double Mach reflection problem.
 
 Ablation and runtime studies identify the source of the observed efficiency gains. Although the OE correction is important for the shock-interaction tests considered here, it is not the dominant measured cost. In matched-order and same-$\mathbb{P}^1$-state comparisons, the reduced complete-run wall-clock times are primarily attributable to replacing stagewise nonlinear HWENO-type reconstructions by compact linear reconstructions and applying OE only once per time step. The standard BP/PP limiter is retained only as an auxiliary admissibility safeguard in the explicitly indicated tests.
 
 The present study also leaves several directions open. The analysis is currently limited to the linear-advection backbone and does not constitute a full nonlinear stability theory for the OE-corrected schemes. In addition, the componentwise conservative-variable implementation for systems is supported by the reported numerical evidence, but a general theoretical characterization of when such a treatment is sufficient remains to be developed. Extending the compact reconstruction and end-of-step OE strategy beyond uniform Cartesian grids, and adapting it to systems with more complex physical constraints, will be important topics for future work.

%\clearpage
\begin{appendix}

\section{How the compact linear reconstruction formulas are obtained}\label{app:derivation}
The main text explains the stencil constraints used by the FAM reconstruction. The algebra behind the explicit formulas is elementary but useful to record once: one writes the target polynomial in a local orthogonal basis, enforces the available cell-average and first-order-moment conditions on the chosen compact stencil, and solves the resulting square linear system once and for all. The implementation then evaluates the precomputed linear combinations listed below; machine-readable versions of the jump and Fourier matrices are included as ancillary text files in this source package.

In one dimension, the 3rd- and 4th-order reconstructions come from a cubic ansatz determined by three neighboring cell averages and the central first-order moment, while the 5th- and 6th-order reconstructions come from a quintic ansatz determined by three neighboring cell averages and three first-order moments. In two dimensions, the 3rd- and 4th-order reconstructions come from an incomplete quartic ansatz on a $3\times3$ patch, and the 5th- and 6th-order reconstructions come from an incomplete sixth-degree ansatz on the same patch using the directional first-order moments indicated in Section~\ref{app1}. The formulas recorded in Sections~\ref{app:1dcoeff} and \ref{app1} are the symbolic solutions of those square moment-matching systems.

As a concrete illustration, consider the one-dimensional cubic auxiliary
reconstruction used for FAM3/FAM4,
$Q(x)=\sum_{\ell=0}^{3}c_\ell\phi_i^{(\ell)}(x).$
It is determined by matching the cell averages on $I_{i-1}$, $I_i$, and
$I_{i+1}$, together with the first moment on the central cell. The central
constraints give
$c_0=u_i^{(0)}, c_1=12u_i^{(1)}.$
The remaining two constraints, obtained from the averages over the neighboring
cells, are
$c_0-c_1+c_2-\frac{11}{10}c_3=u_{i-1}^{(0)},
c_0+c_1+c_2+\frac{11}{10}c_3=u_{i+1}^{(0)}.$
Thus,
$c_0=u_i^{(0)},
c_1=12u_i^{(1)},
c_2=\frac12\left(u_{i-1}^{(0)}-2u_i^{(0)}+_{i+1}^{(0)}\right),
c_3=\frac{5}{11}\left(u_{i+1}^{(0)}-u_{i-1}^{(0)}\right)
-\frac{120}{11}u_i^{(1)}.$
The same moment-matching procedure leads to the higher-order coefficient
formulas on larger stencils.

\section{One-dimensional reconstruction coefficients}\label{app:1dcoeff}
For completeness, we record the explicit coefficients used by the one-dimensional reconstructions introduced in Section~2.4 of the main paper. We use the local orthogonal basis
\[
\phi_i^{(0)}(x)=1,
\phi_i^{(1)}(x)=\xi_i:=\frac{x-x_i}{\Delta x},
\phi_i^{(2)}(x)=\xi_i^2-\frac{1}{12},
\]
\[
\phi_i^{(3)}(x)=\xi_i^3-\frac{3}{20}\xi_i,
\phi_i^{(4)}(x)=\xi_i^4-\frac{3}{14}\xi_i^2+\frac{3}{560},
\phi_i^{(5)}(x)=\xi_i^5-\frac{5}{18}\xi_i^3+\frac{5}{336}\xi_i.
\]

For FAM3/FAM4, let
$Q(x)=\sum_{\ell=0}^{3}c_{\ell}\phi_i^{(\ell)}(x),$
with
$c_0=u_i^{(0)},\qquad c_1=12u_i^{(1)},c_2=\frac12\big(u_{i-1}^{(0)}-2u_i^{(0)}+u_{i+1}^{(0)}\big),c_3=\frac{5}{11}\big(u_{i+1}^{(0)}-u_{i-1}^{(0)}\big)-\frac{120}{11}u_i^{(1)}.$ Then $p_2(x)=\sum_{\ell=0}^{2}c_{\ell}\phi_i^{(\ell)}(x)$ gives the
third-order reconstruction, while $p_3(x)=Q(x)$ gives the fourth-order one.

For FAM5/FAM6, let $Q(x)=\sum_{\ell=0}^{5}c_{\ell}\phi_i^{(\ell)}(x),$
where $c_0=u_i^{(0)}$, $c_1=12u_i^{(1)}$, and
\begin{equation*}
\begin{aligned}
&c_2=\frac{73}{56}(u^{(0)}_{i-1}-2u^{(0)}_{i}+u^{(0)}_{i+1})+\frac{135}{28}(u^{(1)}_{i-1}-u^{(1)}_{i+1}),
\\
&c_3=\frac{595}{324}(u^{(0)}_{i+1}-u^{(0)}_{i-1})-\frac{985}{162}(u^{(1)}_{i-1}+u^{(1)}_{i+1})-\frac{2585}{81}u^{(1)}_{i},
\\
&c_4=-\frac{5}{8}(u^{(0)}_{i-1}-2u^{(0)}_{i}+u^{(0)}_{i+1})-\frac{15}{4}(u^{(1)}_{i-1}-u^{(1)}_{i+1}),
\\
&c_5=\frac{35}{36}(u^{(0)}_{i-1}-u^{(0)}_{i+1})+\frac{77}{18}(u^{(1)}_{i-1}+u^{(1)}_{i+1})+\frac{133}{9}u^{(1)}_{i}.
\end{aligned}
\end{equation*}
Then $p_4(x)=\sum_{\ell=0}^{4}c_{\ell}\phi_i^{(\ell)}(x)$ gives the fifth-order reconstruction, while $p_5(x)=Q(x)$ gives the sixth-order one.

\section{One-dimensional OE jump formulas}\label{app:1djump}
The one-dimensional OE step in the main paper uses jumps of the reconstructed solution at cell interfaces. For FAM$k$, $k=3,4,5,6$, those jumps are the following explicit linear combinations of the moments after the final RK stage.

\begin{table}[t]
\centering
\caption{Interface jumps $\jump{u_h^*}_{i+\frac12}$ for the one-dimensional OE update,	with $u_i^{(0)}=u_i^{n,(0)}$ and $u_i^{(1)}=u_i^{n,(1)}$.}
\label{tab:supp:jump1}
\small 
\begin{tabular}{c|l}
\toprule[1.2pt]
$k$ & $\jump{u_h^*}_{i+\frac12}$ \\
\midrule
$3$ & $\displaystyle \frac{1}{12}\left(u^{(0)}_{i+2}-u^{(0)}_{i-1}+9(u^{(0)}_{i+1}-u^{(0)}_{i})-72(u^{(1)}_{i+1}+u^{(1)}_{i})\right)$ \\
\midrule
$4$ & $\displaystyle \frac{1}{33}\left(2(u^{(0)}_{i+2}-u^{(0)}_{i-1})+24(u^{(0)}_{i+1}-u^{(0)}_{i})-180(u^{(1)}_{i+1}+u^{(1)}_{i})\right)$ \\
\midrule
$5$ & $\displaystyle \frac{1}{1296}\left(151(u^{(0)}_{i+2}-u^{(0)}_{i-1})+367(u^{(0)}_{i+1}-u^{(0)}_{i})-578(u^{(1)}_{i-1}+u^{(1)}_{i+2})-4342(u^{(1)}_{i+1}+u^{(1)}_{i})\right)$ \\
\midrule
$6$ & $\displaystyle \frac{1}{108}\left(13(u^{(0)}_{i+2}-u^{(0)}_{i-1})+31(u^{(0)}_{i+1}-u^{(0)}_{i})-50(u^{(1)}_{i+2}+u^{(1)}_{i-1})-370(u^{(1)}_{i+1}+u^{(1)}_{i})\right)$ \\
\bottomrule[1.2pt]
\end{tabular}
\end{table}

\section{Additional smooth-problem data}\label{app:add-smooth}
Tables~\ref{tab:supp:1dBurMoment} and \ref{tab:supp:1dEulerMoment} record the first-order-moment errors for the smooth Burgers and Euler test of the main paper. They confirm the expected $k$th-order convergence of the evolved first-order moments and complement the cell-average tables reported in the main text.

\begin{table}[t]
\centering
\caption{$L^1$ errors of the first-order moment for the smooth Burgers test of the main paper.}
\label{tab:supp:1dBurMoment}
\small 
\begin{tabular}{c|cccccccc}
\toprule[1.2pt]
$N_x$ & \multicolumn{2}{c}{FAM3} & \multicolumn{2}{c}{FAM4} & \multicolumn{2}{c}{FAM5} & \multicolumn{2}{c}{FAM6} \\
\cmidrule(l){2-3}\cmidrule(l){4-5}\cmidrule(l){6-7}\cmidrule(l){8-9}
& $L^1$ error & order & $L^1$ error & order & $L^1$ error & order & $L^1$ error & order \\
\midrule
30 & 1.50E-05 & --   & 5.54E-06 & --   & 5.40E-07 & --   & 4.99E-07 & --   \\
60 & 2.37E-06 & 2.67 & 4.46E-07 & 3.64 & 1.15E-08 & 5.55 & 8.85E-09 & 5.82 \\
90 & 7.51E-07 & 2.83 & 8.93E-08 & 3.96 & 1.08E-09 & 5.83 & 7.59E-10 & 6.06 \\
120& 3.28E-07 & 2.88 & 2.85E-08 & 3.97 & 2.09E-10 & 5.71 & 1.34E-10 & 6.04 \\
150& 1.71E-07 & 2.93 & 1.18E-08 & 3.96 & 5.84E-11 & 5.72 & 3.62E-11 & 5.86 \\
180& 9.97E-08 & 2.94 & 5.71E-09 & 3.97 & 2.08E-11 & 5.65 & 1.29E-11 & 5.67 \\
\bottomrule[1.2pt]
\end{tabular}
\end{table}

\begin{table}[htpb]
	\centering
	\caption{$L^1$ errors of the first-order moment for the smooth 1D Euler test of the main paper.}
	\label{tab:supp:1dEulerMoment}
	\small
		\begin{tabular}{c|cccccccc}
			\toprule[1.2pt]
			$N_x$ & \multicolumn{2}{c}{FAM3} & \multicolumn{2}{c}{FAM4} & \multicolumn{2}{c}{FAM5} & \multicolumn{2}{c}{FAM6} \\
			\cmidrule(l){2-3}\cmidrule(l){4-5}\cmidrule(l){6-7}\cmidrule(l){8-9}
			& $L^1$ error & order & $L^1$ error & order & $L^1$ error & order & $L^1$ error & order \\
			\midrule
			20&	 5.47E-06&    -- &	9.82E-07&    -- &	4.79E-09&    -- &	2.63E-09&    -- 	\\
			30&	 1.61E-06&   3.01&	1.79E-07&   4.20&	5.30E-10&   5.43&	2.19E-10&   6.13	\\
			40&	 6.82E-07&   2.99&	5.43E-08&   4.14&	1.18E-10&   5.23&	3.79E-11&   6.09	\\
			50&	 3.50E-07&   2.99&	2.18E-08&   4.10&	3.72E-11&   5.16&	9.82E-12&   6.06	\\
			60&	 2.02E-07&   3.00&	1.04E-08&   4.07&	1.47E-11&   5.10&	3.26E-12&   6.04	\\
			70&	 1.28E-07&   2.99&	5.55E-09&   4.05&	6.70E-12&   5.08&	1.29E-12&   6.02	\\
			\bottomrule[1.2pt]
	\end{tabular}
\end{table}

\section{Supplementary remark on the Type~II reference quantity}\label{app:typeII}
The Type~II OE coefficient uses the unnormalized local sum $S_{\Lambda}=\sum_{p\in\Lambda}u_p^{n,(0)}$
instead of the more conventional local average. This is intentional rather than a typographical shortcut. A simple heuristic explains why. If the stencil values take the form
$u_p=c+\varepsilon v_p, p\in\Lambda,$
with a large background state $c$ and a comparatively small local fluctuation $\varepsilon v_p$, then
\[
S_{\Lambda}=|\Lambda|c+\varepsilon\sum_{q\in\Lambda}v_q,\qquad
u_p-S_{\Lambda}=(1-|\Lambda|)c+\varepsilon\Big(v_p-\sum_{q\in\Lambda}v_q\Big).
\]
Hence the denominator in the Type~II coefficient becomes large when a small feature is superposed on a large background state, and the resulting damping is weakened. This is precisely the behavior desired in the multiscale tests of the main paper, where an affine-invariant coefficient can over-damp the small-scale structures. The price is that Type~II is no longer globally affine invariant. We therefore interpret it as a \emph{background-aware, locally scale-invariant} choice rather than an affine-invariant one.

\section{Additional benchmark descriptions and figures}\label{app:extra-bench}
The examples in this section provide supporting evidence for the numerical discussion in Section~4. They help clarify the role of the OE coefficients and the robustness of the overall workflow on problems that are secondary to the main benchmark set.

\paragraph{Discontinuous linear advection}
We solve the linear advection equation $\partial_t u+\partial_x u=0$ on $[-1,1]$ with periodic boundaries. The initial profile is the standard composite test consisting of Gaussians, a square wave, a sharp triangle, and a half ellipse,   Ref.~[14] in the main text. We initialize the computation with $u(x,0)=\lambda u_0(x)+c$ using $(\lambda,c)=(0.01,1)$ and evolve to one period $T=2$. Figure~\ref{fig:supp:1dAdvec} plots the rescaled numerical cell averages $(u-c)/\lambda$ for FAM3 and FAM5. In this linear-advection setting, the Type~I OE correction preserves the expected global scale/translation behavior much better than Type~II, whose locally scaled denominator can generate visible oscillations near the discontinuities.

\begin{figure}[t]
\centering
\includegraphics[width=0.68\textwidth]{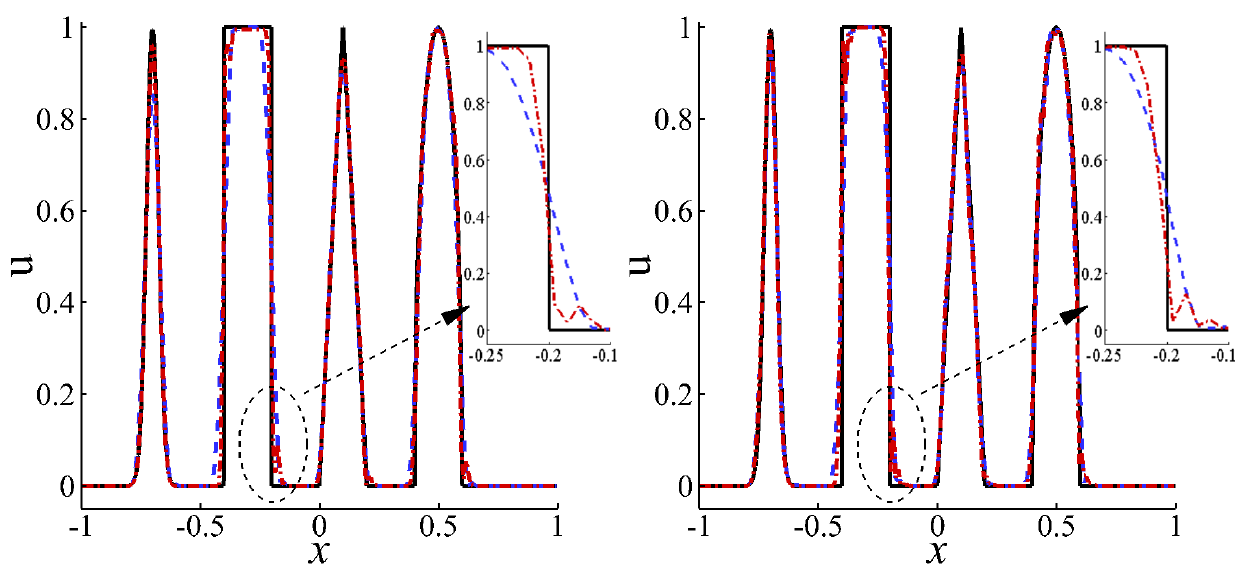}
\caption{Rescaled cell averages for the discontinuous linear-advection test at $T=2$ on 200 cells. Left: FAM3. Right: FAM5. Blue dashed: Type~I OE; red dash-dotted: Type~II OE.}
\label{fig:supp:1dAdvec}
\end{figure} 

\paragraph{Woodward--Colella blast wave} 
We simulated the Woodward--Colella blast wave problem on $[0,1]$ with reflective boundaries and initial data $\bm u(x,0)=\bm u_0(x)$, where $(\rho_0,\mu_0,p_0)=(1,0,10^3)$ for $0\le x\le0.1$, $(1,0,10^{-2})$ for $0.1<x\le0.9$, and $(1,0,10^2)$ for $0.9<x\le1$  with the final time $T=0.038$. Figure~\ref{fig:supp:blastwave} reports FAM density profiles for this highly compressive test. As in the main text, these auxiliary Euler simulations are conducted componentwise in conservative variables; they are supporting examples rather than a separate robustness theorem. The black curve denotes the reference solution from the FAM5 scheme using 8000 cells.

\begin{figure}[t]
\centering
\includegraphics[width=0.86\textwidth]{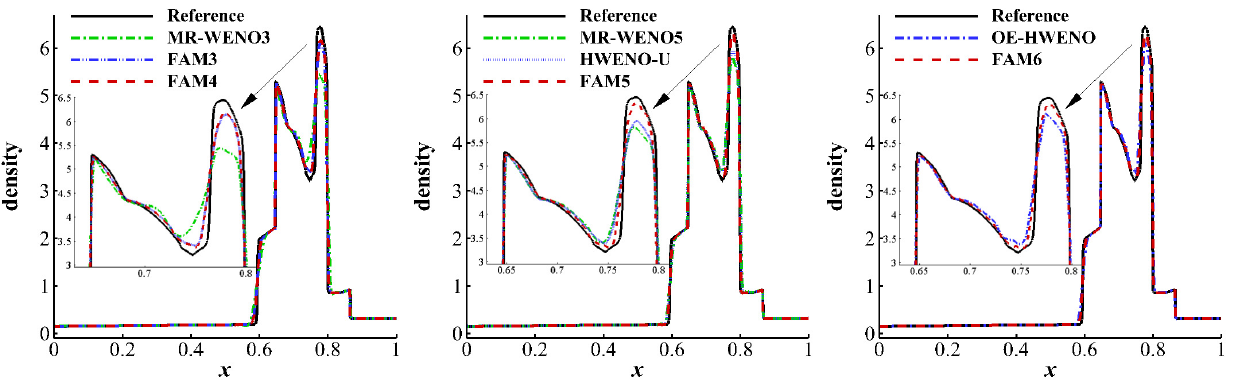}
\caption{Density for the Woodward--Colella blast wave problem on 800 cells. }
\label{fig:supp:blastwave}
\end{figure}

\paragraph{Mach~2000 jet}
We also record the standard Mach~2000 astrophysical jet problem for the two-dimensional Euler equations with $\gamma=5/3$. The computational domain is $[0,1]\times[-0.25,0.25]$. The initial ambient state is $(\rho,\mu,\nu,p)=(0.5,0,0,0.4127)$. On the left boundary, a jet inflow $(5,800,0,0.4127)$ is prescribed for $y\in[-0.05,0.05]$, while the ambient state is imposed on the rest of the left boundary; the right, top, and bottom boundaries use outflow conditions. The snapshot in Figure~\ref{fig:supp:hm2000full} is taken at $T=0.001$ on a $600\times300$ mesh. Its role is qualitative: it provides another multiscale system problem on which the Type~II coefficient gives cleaner small-scale structures than Type~I.
	  
\begin{figure}[!htbp]
\centering
\includegraphics[width=0.97\textwidth]{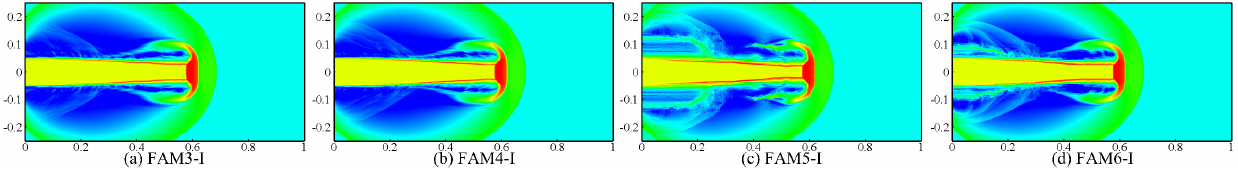}\\
\includegraphics[width=0.97\textwidth]{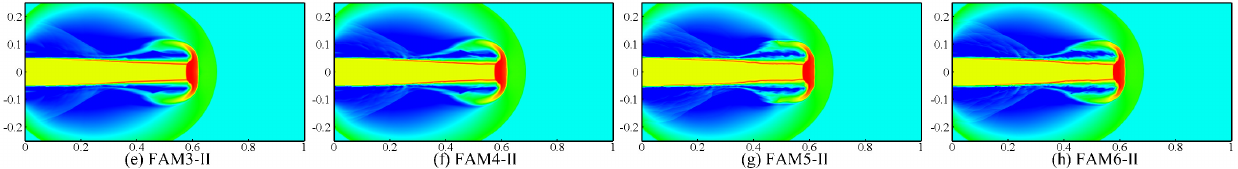}
\caption{Full Type~I/Type~II density galleries for the Mach~2000 jet test described in Appendix F. Top: Type~I OE. Bottom: Type~II OE.} 
\label{fig:supp:hm2000full}
\end{figure}

\section{Supporting sampled CFL curves}\label{app:runtime-cfl}
This section provides the sampled CFL curves corresponding to Proposition 3.4 of the main text. 

\begin{figure}[!htbp]
\centering
\includegraphics[width=0.82\textwidth]{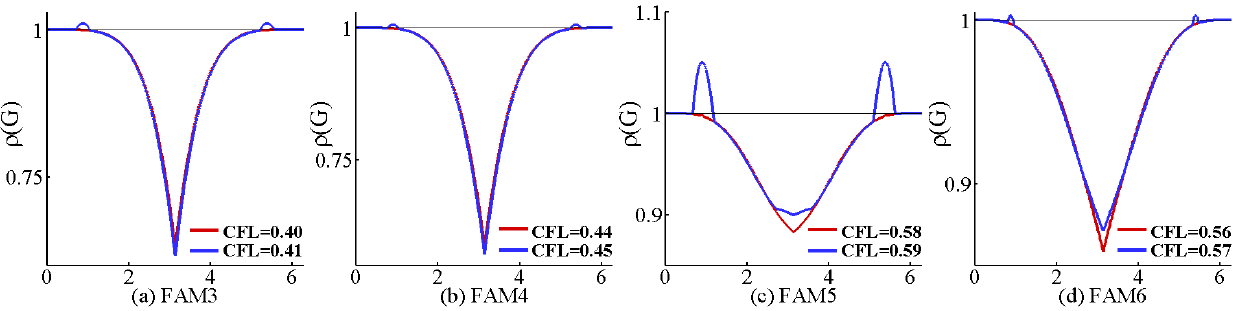}
\caption{Sampled spectral radius of the SSP-RK3 amplification matrix for FAM$k$, $k=3,4,5,6$, used to identify the sampled admissible CFL estimates reported in the main paper.}
\label{fig:supp:CFL}
\end{figure}
 
\section{Parameter sensitivity study}\label{app:add-num}
This section supports the parameter choices in the main paper. Figure~\ref{fig:supp:blastCk} shows the sensitivity of the blast wave profile to the constants $\gamma_k$ used in the OE coefficient. The solution changes only mildly over a representative range around the values adopted in the main paper, which helps justify that the reported choices are representative rather than finely tuned. The sampled CFL curves are reported in Section~\ref{app:runtime-cfl}.

\begin{figure}[!htbp]
\centering
\includegraphics[width=0.8\textwidth]{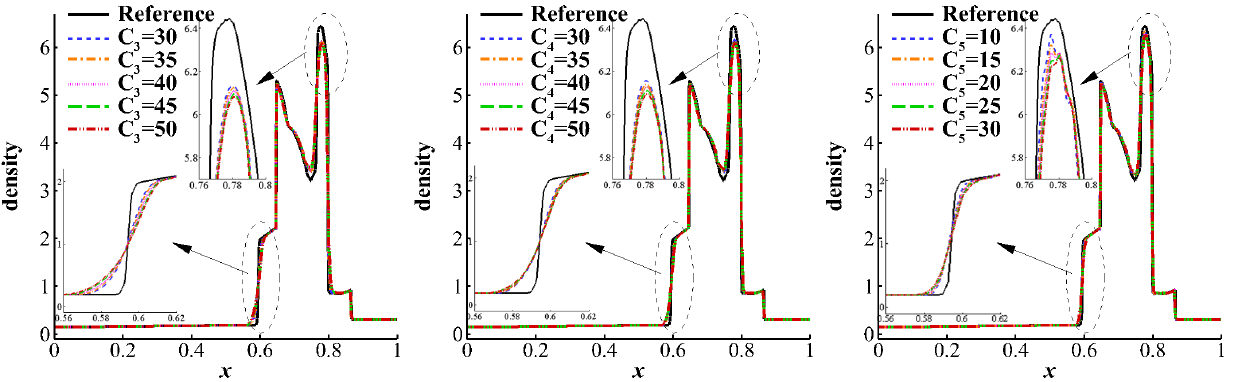}
\caption{Sensitivity of the blast wave solution to the OE constants $\gamma _k$. All runs use Type II OE and the PP limiter.}
\label{fig:supp:blastCk} 
\end{figure}

\section{Supplementary figures for OE ablation and runtime attribution}
\begin{figure}[!htbp]
	\centering
	\includegraphics[width=0.65\textwidth]{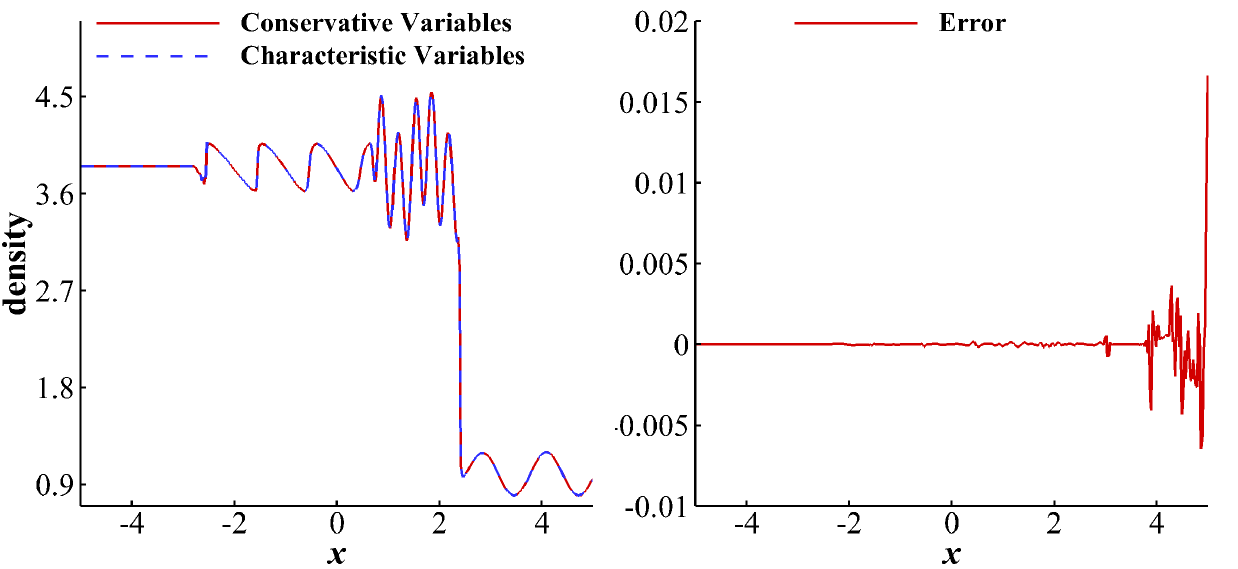}
	\caption{FAM5 density profiles for the Shu--Osher problem on 400 cells, obtained using componentwise conservative-variable reconstruction and local characteristic-variable reconstruction. The corresponding errors $\rho_{\rm cons} - \rho_{\rm char}$ are shown on the right.}
	\label{fig:supp:ShuOsher}
\end{figure}
This section provides supporting evidence for the numerical discussion in Section~4.3 of the main paper. Figures~\ref{fig:supp:ShuOsher} and~\ref{fig:supp1:DoubleMach} compare the FAM5 results obtained with conservative-variable and characteristic-variable reconstructions for the Shu--Osher and double Mach reflection problems, respectively. The corresponding error distributions are also shown. In both tests, the two reconstruction strategies lead to only minor differences, supporting the observation that the choice between conservative and characteristic variables does not materially affect the reported results.

Figure~\ref{fig:supp2:DoubleMach} further compares the FAM results without the OE procedure and with the type-I/type-II OE procedures. The results show that the OE procedure effectively suppresses spurious oscillations, with the type-II OE procedure producing slightly sharper flow structures in the more complex multiscale flow.

\begin{figure}[!htbp]
	\centering
	\includegraphics[width=0.93\textwidth]{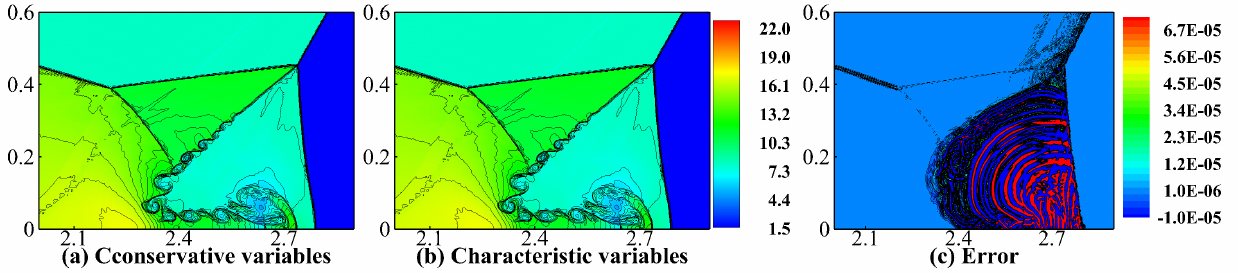}
	\caption{FAM5 density contours for the double Mach reflection problem, obtained using componentwise conservative-variable reconstruction and local characteristic-variable reconstruction. The corresponding error distributions $\rho_{\rm cons} - \rho_{\rm char}$ are shown on the right.}
	\label{fig:supp1:DoubleMach}
\end{figure}

\begin{figure}[!htbp]
	\centering
	\includegraphics[width=0.9\textwidth]{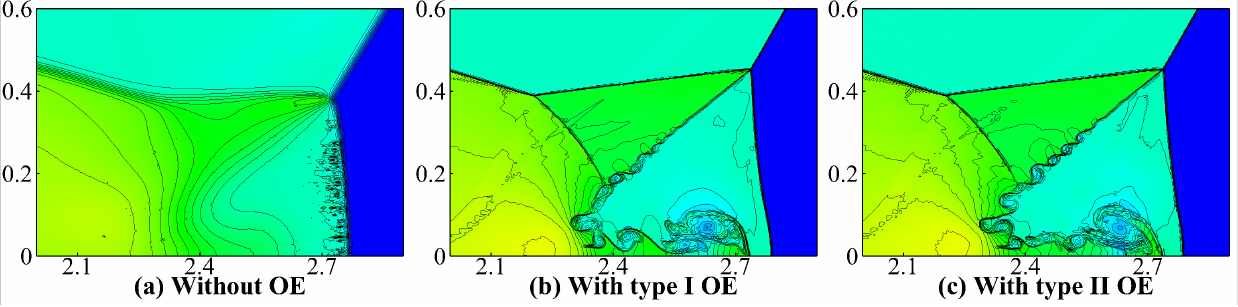}
	\caption{FAM5 density contours for the double Mach reflection problem on a $1440\times480$ mesh, computed without OE and with the Type~I and Type~II OE procedures. The no-OE computation employs the standard positivity-preserving limiter to maintain admissibility in this ablation test.}
	\label{fig:supp2:DoubleMach}
\end{figure}

\section{Full asymptotic expansions for the Fourier analysis}\label{app:fourier-exp}
For completeness we record below the full symbolic expansions of the physical and spurious modes underlying the leading formulas quoted in Section~3 of the main paper:
\[\begin{aligned}
&\bullet\ k=3:\quad
\zeta_1=-\mathrm i\omega-\frac{\mathrm i\omega^5}{720}\Delta x^4+\mathcal O(\Delta x^5),
\zeta_2=-\frac{6}{\Delta x}+\mathrm i\mathcal O(1),
\\&
\bm V_1-\hat{\bm u}_{\rm ex}=
\begin{pmatrix}
\frac{\omega^4}{720}\Delta x^4+\mathcal O(\Delta x^5)\\[0.2ex]
-\frac{\mathrm i\omega^3}{720}\Delta x^3+\mathcal O(\Delta x^4)
\end{pmatrix},
\bm V_2=
\begin{pmatrix}
-\frac{\omega^4}{720}\Delta x^4+\mathcal O(\Delta x^5)\\[0.2ex]
\frac{\mathrm i\omega^3}{720}\Delta x^3+\mathcal O(\Delta x^4)
\end{pmatrix},
\\
&\bullet\ k=4:\quad
\zeta_1=-\mathrm i\omega-\frac{11\omega^6}{7200}\Delta x^5+\mathcal O(\Delta x^6),
\quad
\zeta_2=-\frac{60}{11\Delta x}+\mathrm i\mathcal O(1),
\\
&\hspace{1.95cm}
\bm V_1-\hat{\bm u}_{\rm ex}=
\begin{pmatrix}
-\frac{11\mathrm i\omega^5}{7200}\Delta x^5+\mathcal O(\Delta x^6)\\[0.2ex]
-\frac{11\omega^4}{7200}\Delta x^4+\mathcal O(\Delta x^5)
\end{pmatrix},
\quad
\bm V_2=
\begin{pmatrix}
\frac{11\mathrm i\omega^5}{7200}\Delta x^5+\mathcal O(\Delta x^6)\\[0.2ex]
\frac{11\omega^4}{7200}\Delta x^4+\mathcal O(\Delta x^5)
\end{pmatrix},
\\
&\bullet\ k=5:\quad
\zeta_1=-\mathrm i\omega+\frac{\mathrm i\omega^7}{114800}\Delta x^6+\mathcal O(\Delta x^7),
\quad
\zeta_2=-\frac{205}{54\Delta x}+\mathrm i\mathcal O(1),
\\
&\hspace{1.95cm}
\bm V_1-\hat{\bm u}_{\rm ex}=
\begin{pmatrix}
-\frac{\omega^6}{114800}\Delta x^6+\mathcal O(\Delta x^7)\\[0.2ex]
\frac{\mathrm i\omega^5}{114800}\Delta x^5+\mathcal O(\Delta x^6)
\end{pmatrix},
\quad
\bm V_2=
\begin{pmatrix}
\frac{\omega^6}{114800}\Delta x^6+\mathcal O(\Delta x^7)\\[0.2ex]
-\frac{\mathrm i\omega^5}{114800}\Delta x^5+\mathcal O(\Delta x^6)
\end{pmatrix},
\\
&\bullet\ k=6:\quad
\zeta_1=-\mathrm i\omega-\frac{\omega^8}{19600}\Delta x^7+\mathcal O(\Delta x^8),
\quad
\zeta_2=-\frac{35}{9\Delta x}+\mathrm i\mathcal O(1),
\\
&\hspace{1.95cm}
\bm V_1-\hat{\bm u}_{\rm ex}=
\begin{pmatrix}
-\frac{\mathrm i\omega^7}{19600}\Delta x^7+\mathcal O(\Delta x^8)\\[0.2ex]
-\frac{\omega^6}{19600}\Delta x^6+\mathcal O(\Delta x^7)
\end{pmatrix},
\quad
\bm V_2=
\begin{pmatrix}
\frac{\mathrm i\omega^7}{19600}\Delta x^7+\mathcal O(\Delta x^8)\\[0.2ex]
\frac{\omega^6}{19600}\Delta x^6+\mathcal O(\Delta x^7)
\end{pmatrix}.
\end{aligned}
\]

\section{Constant matrices used in the one-dimensional Fourier symbol}\label{app2b}
\begin{table}[H]
	\centering
	\belowrulesep=0pt
	\aboverulesep=0pt
	\caption{Constant matrices ${\bf D}_{-2}^{(k)}$, ${\bf D}_{-1}^{(k)}$, ${\bf D}_{0}^{(k)}$, and ${\bf D}_{1}^{(k)}$ used in the one-dimensional semi-discrete Fourier symbol for $k=3,4,5,6$.}\label{tab_Dk}
	\renewcommand{\arraystretch}{1.5}
	\setlength{\arraycolsep}{1.5pt}
	\setlength{\tabcolsep}{1.5pt}
	\begin{tabular}{c|c|c|c|c}
		\toprule[1.0pt]
		$k$& ${\bf D}_{-2}^{(k)}$ & ${\bf D}_{-1}^{(k)}$ & ${\bf D}_{0}^{(k)}$ & ${\bf D}_{1}^{(k)}$ \\
		\midrule
		$3$
		&$\begin{bmatrix}\frac{1}{12}&0\\[1pt]-\frac{1}{24}&0\end{bmatrix}$
		&$\begin{bmatrix}\frac{3}{4}&6\\[1pt]-\frac{11}{24}&-3\end{bmatrix}$
		&$\begin{bmatrix}-\frac{3}{4}&-6\\[1pt]\frac{13}{24}&-3\end{bmatrix}$
		&$\begin{bmatrix}-\frac{1}{12}&0\\[1pt]-\frac{1}{24}&0\end{bmatrix}$
		\\\midrule
		$4$
		&$\begin{bmatrix}\frac{2}{33}&0\\[1pt]-\frac{1}{33}&0\end{bmatrix}$
		&$\begin{bmatrix}\frac{17}{22}&\frac{60}{11}\\[1pt]-\frac{59}{132}&-\frac{30}{11}\end{bmatrix}$
		&$\begin{bmatrix}-\frac{8}{11}&-\frac{60}{11}\\[1pt]\frac{35}{66}&-\frac{30}{11}\end{bmatrix}$
		&$\begin{bmatrix}-\frac{7}{66}&0\\[1pt]-\frac{7}{132}&0\end{bmatrix}$
		\\\midrule
		$5$
		&$\begin{bmatrix}\frac{151}{1296}&\frac{289}{648}\\[1pt]-\frac{151}{2592}&-\frac{289}{1296}\end{bmatrix}$
		&$\begin{bmatrix}\frac{605}{1296}&\frac{95}{24}\\[1pt]-\frac{907}{2592}&-\frac{3143}{1296}\end{bmatrix}$
		&$\begin{bmatrix}-\frac{367}{1296}&-\frac{131}{24}\\[1pt]\frac{1447}{2592}&-\frac{2171}{1296}\end{bmatrix}$
		&$\begin{bmatrix}-\frac{389}{1296}&\frac{683}{648}\\[1pt]-\frac{389}{2592}&\frac{683}{1296}\end{bmatrix}$
		\\\midrule
		$6$
		&$\begin{bmatrix}\frac{13}{108}&\frac{25}{54}\\[1pt]-\frac{13}{216}&-\frac{25}{108}\end{bmatrix}$
		&$\begin{bmatrix}\frac{25}{54}&4\\[1pt]-\frac{19}{54}&-\frac{133}{54}\end{bmatrix}$
		&$\begin{bmatrix}-\frac{31}{108}&-\frac{11}{2}\\[1pt]\frac{121}{216}&-\frac{185}{108}\end{bmatrix}$
		&$\begin{bmatrix}-\frac{8}{27}&\frac{28}{27}\\[1pt]-\frac{4}{27}&\frac{14}{27}\end{bmatrix}$
		\\
		\bottomrule[1.0pt]
	\end{tabular}
\end{table}

\section{Orthogonal basis functions and polynomial coefficients for the 2D reconstruction}\label{app1}

The following local basis is orthogonal, but not normalized, on the central cell
$I_{i,j}$. Let
$\xi_i=\frac{x-x_i}{\Delta x}, 
\eta_j=\frac{y-y_j}{\Delta y}.$
The basis functions used in the two-dimensional reconstructions are
	\begin{align*}
		&\phi_{i,j}^{(0)}(x,y)=1,~
		\phi_{i,j}^{(1)}(x,y)=\xi_i:= \frac{x-x_i}{\Delta x},~
		\phi_{i,j}^{(2)}(x,y)=\eta_j:=\frac{y-y_j}{\Delta y},
		\\&
		\phi_{i,j}^{(3)}(x,y)=\xi_i^2-\frac{1}{12},~
		\phi_{i,j}^{(4)}(x,y)=\xi_i\eta_j,~
		\phi_{i,j}^{(5)}(x,y)=\eta_j^2-\frac{1}{12},
		\\&
		\phi_{i,j}^{(6)}(x,y)=\xi_i^3-\frac{3}{20}\xi_i, \quad 
		\phi_{i,j}^{(7)}(x,y)=\phi_{i,j}^{(3)}(x,y)\eta_j, \quad 
		\phi_{i,j}^{(8)}(x,y)=\xi_i\phi_{i,j}^{(5)}(x,y),
		\\&
		\phi_{i,j}^{(9)}(x,y)=\eta_j^3-\frac{3}{20}\eta_j, \quad 
		\phi_{i,j}^{(10)}(x,y)=\phi_{i,j}^{(3)}(x,y)\phi_{i,j}^{(5)}(x,y),
		\\&
		\phi_{i,j}^{(11)}(x,y)=\xi_i^4-\frac{3}{14}\xi_i^2+\frac{3}{560}, \quad 
		\phi_{i,j}^{(12)}(x,y)=\phi_{i,j}^{(6)}(x,y)\eta_j, \quad 
		\phi_{i,j}^{(13)}(x,y)=\xi_i\phi_{i,j}^{(9)}(x,y),
		\\&
		\phi_{i,j}^{(14)}(x,y)=\eta_j^4-\frac{3}{14}\eta_j^2+\frac{3}{560}, 
		\quad 
		\phi_{i,j}^{(15)}(x,y)=\xi_i^5-\frac{5}{18}\xi_i^3+\frac{5}{336}\xi_i,~
		\\&\phi_{i,j}^{(16)}(x,y)=\phi_{i,j}^{(11)}(x,y)\eta_j,~
		\phi_{i,j}^{(17)}(x,y)=\phi_{i,j}^{(6)}(x,y)\phi_{i,j}^{(5)}(x,y),
		\\ & 
		\phi_{i,j}^{(18)}(x,y)=\phi_{i,j}^{(3)}(x,y)\phi_{i,j}^{(9)}(x,y),~
		\phi_{i,j}^{(19)}(x,y)=\xi_i\phi_{i,j}^{(14)}(x,y),~
		\phi_{i,j}^{(20)}(x,y)=\eta_j^5-\frac{5}{18}\eta_j^3+\frac{5}{336}\eta_j,
		\\&
		\phi_{i,j}^{(21)}(x,y)=\phi_{i,j}^{(11)}(x,y)\phi_{i,j}^{(5)}(x,y),~
		\phi_{i,j}^{(22)}(x,y)=\phi_{i,j}^{(3)}(x,y)\phi_{i,j}^{(14)}(x,y).
	\end{align*}
We emphasize that the auxiliary polynomials constructed below are not used
directly as the final reconstruction polynomials. Instead, the final FAM
reconstructions are obtained by projecting the auxiliary polynomials onto
complete polynomial spaces of the desired total degrees.

For FAM3 and FAM4, we first construct an auxiliary incomplete quartic polynomial
$Q_4(x,y)=\sum_{\ell=0}^{10}c_\ell \phi_{i,j}^{(\ell)}(x,y)$ with the coefficients  
\begin{equation*} 
	\begin{aligned}
		&c_{0}=u^{(0)}_{{5}},
		c_{1}=12u^{(1)}_{{5}},
		c_{2}=12u^{(2)}_{{5}},
		c_{3}=\frac{1}{2}\big(u^{(0)}_4-2u^{(0)}_5+u^{(0)}_6\big),
		\\&
		c_{4}=\frac{1}{4}\big(u^{(0)}_1-u^{(0)}_3-u^{(0)}_7+u^{(0)}_9\big),
		c_{5}=\frac{1}{2}\big(u^{(0)}_2-2u^{(0)}_5+u^{(0)}_8\big), 
		\\&
		c_{6}=\frac{5}{11}\big(-u^{(0)}_4+u^{(0)}_6\big)-\frac{120}{11}u^{(1)}_5,
		c_{7}=\frac{1}{4}\big(-u^{(0)}_1+2u^{(0)}_2-u^{(0)}_3+u^{(0)}_7-2u^{(0)}_8+u^{(0)}_9\big),
		\\&
		c_{8}=\frac{1}{4}\big(-u^{(0)}_1+u^{(0)}_3+2u^{(0)}_4-2u^{(0)}_6-u^{(0)}_7+u^{(0)}_9\big),
		c_{9}=\frac{5}{11}\big(-u^{(0)}_2+u^{(0)}_8\big)-\frac{120}{11}u^{(2)}_5,
		\\&
		c_{10}=\frac{1}{4}\big(
		u^{(0)}_1-2u^{(0)}_2+u^{(0)}_3
		-2u^{(0)}_4+4u^{(0)}_5-2u^{(0)}_6
		+u^{(0)}_7-2u^{(0)}_8+u^{(0)}_9
		\big).
	\end{aligned}
\end{equation*}
The final FAM3 and FAM4 reconstruction polynomials are obtained by projection: $p_3(x,y)=\sum_{\ell=0}^{5}c_\ell \phi_{i,j}^{(\ell)}(x,y)$, $p_4(x,y)=\sum_{\ell=0}^{9}c_\ell \phi_{i,j}^{(\ell)}(x,y).$ 
Thus, $p_3$ is a complete quadratic polynomial and $p_4$ is a complete
cubic polynomial. The coefficient $c_{10}$ is used only in the auxiliary
incomplete quartic polynomial $Q_4$ to close the square moment system.

For FAM5 and FAM6, we construct an auxiliary incomplete degree-six polynomial $	Q_6(x,y)=\sum_{\ell=0}^{22}c_\ell \phi_{i,j}^{(\ell)}(x,y).$
The coefficients of $Q_6(x,y)$ are
\begin{align}
	&c_{0}=u^{(0)}_{{5}},\qquad
	c_{1}=12u^{(1)}_{{5}},\qquad
	c_{2}=12u^{(2)}_{{5}},
	\qquad 
	c_{3}=\frac{73}{56}\big(u^{(0)}_4-2u^{(0)}_5+u^{(0)}_6\big)
	+\frac{135}{28}\big(u^{(1)}_4-u^{(1)}_6\big),
	\nonumber\\&
	c_{4}=\frac{41}{76}\big(u^{(0)}_1-u^{(0)}_3-u^{(0)}_7+u^{(0)}_9\big)
	+\frac{33}{19}\big(
	u^{(1)}_1+u^{(1)}_3-u^{(1)}_7-u^{(1)}_9
	+u^{(2)}_1-u^{(2)}_3+u^{(2)}_7-u^{(2)}_9
	\big),
	\nonumber\\&
	c_{5}=\frac{73}{56}\big(u^{(0)}_2-2u^{(0)}_5+u^{(0)}_8\big)
	+\frac{135}{28}\big(u^{(2)}_2-u^{(2)}_8\big),
	\qquad 
	c_{6}=\frac{595}{324}\big(-u^{(0)}_4+u^{(0)}_6\big)
	-\frac{985}{162}\big(u^{(1)}_4+u^{(1)}_6\big)
	-\frac{2585}{81}u^{(1)}_5,
	\nonumber\\&
	c_{7}=\frac{1695}{2128}\big(-u^{(0)}_1+2u^{(0)}_2-u^{(0)}_3
	+u^{(0)}_7-2u^{(0)}_8+u^{(0)}_9\big)
	\nonumber\\&\quad
	+\frac{135}{56}\big(-u^{(1)}_1+u^{(1)}_3+u^{(1)}_7-u^{(1)}_9\big)
	+\frac{33}{19}\big(-u^{(2)}_1+2u^{(2)}_2-u^{(2)}_3
	-u^{(2)}_7+2u^{(2)}_8-u^{(2)}_9\big),
	\nonumber\\&
	c_{8}=\frac{1695}{2128}\big(-u^{(0)}_1+u^{(0)}_3
	+2u^{(0)}_4-2u^{(0)}_6-u^{(0)}_7+u^{(0)}_9\big)
	\nonumber\\&\quad
	+\frac{33}{19}\big(-u^{(1)}_1-u^{(1)}_3+2u^{(1)}_4+2u^{(1)}_6
	-u^{(1)}_7-u^{(1)}_9\big)
	+\frac{135}{56}\big(-u^{(2)}_1+u^{(2)}_3+u^{(2)}_7-u^{(2)}_9\big),
	\nonumber\\&
	c_{9}=\frac{595}{324}\big(-u^{(0)}_2+u^{(0)}_8\big)
	-\frac{985}{162}\big(u^{(2)}_2+u^{(2)}_8\big)
	-\frac{2585}{81}u^{(2)}_5,
	\nonumber\\&
	c_{10}=\frac{59}{56}\big(u^{(0)}_1-2u^{(0)}_2+u^{(0)}_3
	-2u^{(0)}_4+4u^{(0)}_5-2u^{(0)}_6
	+u^{(0)}_7-2u^{(0)}_8+u^{(0)}_9\big)
	\nonumber\\&\quad
	+\frac{135}{56}\big(
	u^{(1)}_1-u^{(1)}_3-2u^{(1)}_4+2u^{(1)}_6+u^{(1)}_7-u^{(1)}_9
	+u^{(2)}_1-2u^{(2)}_2+u^{(2)}_3-u^{(2)}_7+2u^{(2)}_8-u^{(2)}_9
	\big),
	\nonumber\\&
	c_{11}=\frac{5}{8}\big(-u^{(0)}_4+2u^{(0)}_5-u^{(0)}_6\big)
	+\frac{15}{4}\big(-u^{(1)}_4+u^{(1)}_6\big),
	\nonumber\\&
	c_{12}=\frac{5}{38}\big(-u^{(0)}_1+u^{(0)}_3+u^{(0)}_7-u^{(0)}_9\big)
	+\frac{30}{19}\big(-u^{(1)}_1-u^{(1)}_3+u^{(1)}_7+u^{(1)}_9\big),
	\nonumber\\&
	c_{13}=\frac{5}{38}\big(-u^{(0)}_1+u^{(0)}_3+u^{(0)}_7-u^{(0)}_9\big)
	+\frac{30}{19}\big(-u^{(2)}_1+u^{(2)}_3-u^{(2)}_7+u^{(2)}_9\big),
	\nonumber\\&
	c_{14}=\frac{5}{8}\big(-u^{(0)}_2+2u^{(0)}_5-u^{(0)}_8\big)
	+\frac{15}{4}\big(-u^{(2)}_2+u^{(2)}_8\big),
\qquad
	c_{15}=\frac{35}{36}\big(u^{(0)}_4-u^{(0)}_6\big)
	+\frac{77}{18}\big(u^{(1)}_4+u^{(1)}_6\big)
	+\frac{133}{9}u^{(1)}_5,
	\nonumber\\&
	c_{16}=\frac{5}{16}\big(u^{(0)}_1-2u^{(0)}_2+u^{(0)}_3
	-u^{(0)}_7+2u^{(0)}_8-u^{(0)}_9\big)
	+\frac{15}{8}\big(u^{(1)}_1-u^{(1)}_3-u^{(1)}_7+u^{(1)}_9\big),
	\nonumber\\&
	c_{17}=\frac{5}{38}\big(u^{(0)}_1-u^{(0)}_3-2u^{(0)}_4
	+2u^{(0)}_6+u^{(0)}_7-u^{(0)}_9\big)
	+\frac{30}{19}\big(u^{(1)}_1+u^{(1)}_3-2u^{(1)}_4
	-2u^{(1)}_6+u^{(1)}_7+u^{(1)}_9\big),
	\nonumber\\&
	c_{18}=\frac{5}{38}\big(u^{(0)}_1-2u^{(0)}_2+u^{(0)}_3
	-u^{(0)}_7+2u^{(0)}_8-u^{(0)}_9\big)
	+\frac{30}{19}\big(u^{(2)}_1-2u^{(2)}_2+u^{(2)}_3
	+u^{(2)}_7-2u^{(2)}_8+u^{(2)}_9\big),
	\nonumber\\&
	c_{19}=\frac{5}{16}\big(u^{(0)}_1-u^{(0)}_3-2u^{(0)}_4
	+2u^{(0)}_6+u^{(0)}_7-u^{(0)}_9\big)
	+\frac{15}{8}\big(u^{(2)}_1-u^{(2)}_3-u^{(2)}_7+u^{(2)}_9\big),
	\nonumber\\&
	c_{20}=\frac{35}{36}\big(u^{(0)}_2-u^{(0)}_8\big)
	+\frac{77}{18}\big(u^{(2)}_2+u^{(2)}_8\big)
	+\frac{133}{9}u^{(2)}_5,
	\nonumber\\&
	c_{21}=\frac{5}{16}\big(-u^{(0)}_1+2u^{(0)}_2-u^{(0)}_3
	+2u^{(0)}_4-4u^{(0)}_5+2u^{(0)}_6
	-u^{(0)}_7+2u^{(0)}_8-u^{(0)}_9\big)
	\nonumber\\&\quad
	+\frac{15}{8}\big(-u^{(1)}_1+u^{(1)}_3+2u^{(1)}_4
	-2u^{(1)}_6-u^{(1)}_7+u^{(1)}_9\big),
	\nonumber\\&
	c_{22}=\frac{5}{16}\big(-u^{(0)}_1+2u^{(0)}_2-u^{(0)}_3
	+2u^{(0)}_4-4u^{(0)}_5+2u^{(0)}_6
	-u^{(0)}_7+2u^{(0)}_8-u^{(0)}_9\big)
	\nonumber\\&\quad
	+\frac{15}{8}\big(-u^{(2)}_1+2u^{(2)}_2-u^{(2)}_3
	+u^{(2)}_7-2u^{(2)}_8+u^{(2)}_9\big).\nonumber
\end{align}

The final FAM5 and FAM6 reconstruction polynomials are obtained by projection: $p_5(x,y)=\sum_{\ell=0}^{14}c_\ell \phi_{i,j}^{(\ell)}(x,y),p_5(x,y)=\sum_{\ell=0}^{20}c_\ell \phi_{i,j}^{(\ell)}(x,y).$
Therefore, $p_5$ is a complete quartic polynomial and $p_6$ is a complete
quintic polynomial. For FAM5, the coefficients $c_{15},\ldots,c_{22}$ are
discarded after projection. For FAM6, only the auxiliary degree-six coefficients
$c_{21}$ and $c_{22}$ are discarded after projection. These higher-degree
coefficients are introduced only to close the square moment system for the
auxiliary polynomial $Q_6$.

\section{Two-dimensional jump notation and constant matrices}\label{app2}
For completeness, we repeat the notation behind the two-dimensional OE coefficients. The jumps across the vertical and horizontal interfaces are $\jump{u_h^*}_{i+\frac12,j}=u_h^*(x_{i+\frac12}^+,y_j)-u_h^*(x_{i+\frac12}^-,y_j),
\jump{u_h^*}_{i,j+\frac12}=u_h^*(x_i,y_{j+\frac12}^+)-u_h^*(x_i,y_{j+\frac12}^-).$
For $k=3,4,5,6$, these jumps can be written compactly as
\[
\left\{
\begin{aligned}
\jump{u_h^*}_{i+\frac12,j}&=
\left\langle \bfA_k,{\bf U}^{n,(0)}_x\right\rangle+
\left\langle \bfB_k,{\bf U}^{n,(1)}_x\right\rangle+
\left\langle \bfC_k,{\bf U}^{n,(2)}_x\right\rangle,\\
\jump{u_h^*}_{i,j+\frac12}&=
\left\langle \bfA_k,({\bf U}^{n,(0)}_y)^\top\right\rangle+
\left\langle \bfC_k,({\bf U}^{n,(1)}_y)^\top\right\rangle+
\left\langle \bfB_k,({\bf U}^{n,(2)}_y)^\top\right\rangle,
\end{aligned}
\right.
\qquad k=3,4,5,6,
\]
where $\langle\cdot,\cdot\rangle$ is the Frobenius inner product, and the directional stencil-value matrices are
${\bf U}^{n,(r)}_x=[u_{i+s,j+\ell}^{n,(r)}]\in\mathbb R^{4\times3}, {\bf U}^{n,(r)}_y=[u_{i+\ell,j+s}^{n,(r)}]\in\mathbb R^{3\times4}, r=0,1,2,$
with $\ell=-1,0,1$ and $s=-1,0,1,2$. Tables~\ref{tab:supp:ABC3456} list the constant matrices $\bfA_k$, $\bfB_k$, and $\bfC_k$.

\begin{table}[H]
\centering
\belowrulesep=0pt
\aboverulesep=0pt
\caption{Constant matrices $\bfA_k$, $\bfB_k$, and $\bfC_k$ in the two-dimensional jump representation for $k=3,4,5,6$.}
\label{tab:supp:ABC3456}
\scriptsize
\setlength{\tabcolsep}{1.5pt}
\renewcommand{\arraystretch}{1.35}
\resizebox{\textwidth}{!}{
\begin{tabular}{c|c|c|c}
\toprule[1.2pt]
$k$ & $\bfA_k$ & $\bfB_k$ & $\bfC_k$ \\
\midrule
$3$ &
$\begin{bmatrix}
0&-\frac{1}{12}&0\\[2pt]
\frac{1}{24}&-\frac{5}{6}&\frac{1}{24}\\[2pt]
-\frac{1}{24}&\frac{5}{6}&-\frac{1}{24}\\[2pt]
0&\frac{1}{12}&0
\end{bmatrix}$
&
$\begin{bmatrix}
0&0&0\\[2pt]
0&-6&0\\[2pt]
0&-6&0\\[2pt]
0&0&0
\end{bmatrix}$
&
$\begin{bmatrix}
0&0&0\\[2pt]
0&0&0\\[2pt]
0&0&0\\[2pt]
0&0&0
\end{bmatrix}$
\\ \midrule
$4$ &
$\begin{bmatrix}
-{\frac{1}{96}}&-{\frac{7}{176}}&-{\frac{1}{96}}\\[2pt]
\frac{1}{32}&-{\frac{139}{176}}&\frac{1}{32}\\[2pt]
-\frac{1}{32}&{\frac{139}{176}}&-\frac{1}{32}\\[2pt]
{\frac{1}{96}}&{\frac{7}{176}}&{\frac{1}{96}} 
\end{bmatrix}$
&
$\begin{bmatrix}
0&0&0\\[2pt]
0&-\frac{60}{11}&0\\[2pt]
0&-\frac{60}{11}&0\\[2pt]
0&0&0
\end{bmatrix}$
&
$\begin{bmatrix}
0&0&0\\[2pt]
0&0&0\\[2pt]
0&0&0\\[2pt]
0&0&0
\end{bmatrix}$
\\\midrule
$5$ &
$\begin{bmatrix}
	-{\frac{2843}{153216}}&-{\frac{54745}{689472}}&-{\frac{2843}{153216}}\\[2pt]
	{\frac{297}{8512}}&-{\frac{121679}{344736}}&{\frac{297}{8512}}  \\[2pt]
	-{\frac{297}{8512}}&{\frac{121679}{344736}}&-{\frac{297}{8512}}\\[2pt]
	{\frac{2843}{153216}}&{\frac{54745}{689472}}&{\frac{2843}{153216}} 
\end{bmatrix}$
&
$\begin{bmatrix}
-{\frac{331}{8512}}&-{\frac{126937}{344736}}&-{\frac{331}{8512}}\\[2pt]
-{\frac{901}{8512}}&-{\frac	{1081991}{344736}}&-{\frac{901}{8512}}\\[2pt]
-{\frac{901}{8512}}&-{\frac{1081991}{344736}}&-{\frac{901}{8512}}\\[2pt]
-{\frac{331}{8512}}&-{\frac{126937}{344736}}&-{\frac{331}{8512}}
\end{bmatrix}$
&
$\begin{bmatrix}
-{\frac{15}{224}}&0&{\frac{15}{224}}\\[2pt]
{\frac{99}{448}}&0&-{\frac{99}{448}}\\[2pt]
-{\frac{99}{448}}&0&{\frac{99}{448}}\\[2pt]
{\frac{15}{224}}&0&-{\frac{15}{224}}
\end{bmatrix}$
\\ \midrule
$6$ &
$\begin{bmatrix}
-{\frac{829}{43776}}&-{\frac{5417}{65664}}&-{\frac{829}{43776}}\\[2pt]
{\frac{509}{14592}}&-{\frac{23429}{65664}}&{\frac{509}{14592}}\\[2pt]
-{\frac{509}{14592}}&{\frac{23429}{65664}}&-{\frac{509}{14592}}\\[2pt]
{\frac{829}{43776}}&{\frac{5417}{65664}}&{\frac{829}{43776}} 
\end{bmatrix}$
&
$\begin{bmatrix}
-{\frac{21}{608}}&-{\frac{3233}{8208}}&-{\frac{21}{608}}\\[2pt]
-{\frac{59}{608}}&-{\frac{26527}{8208}}&-{\frac{59}{608}}\\[2pt]
-{\frac{59}{608}}&-{\frac{26527}{8208}}&-{\frac{59}{608}}\\[2pt]
-{\frac{21}{608}}&-{\frac{3233}{8208}}&-{\frac{21}{608}} 
\end{bmatrix}$
&
$\begin{bmatrix}
-{\frac{9}{128}}&0&{\frac{9}{128}}\\[2pt]
{\frac{27}{128}}&0&-{\frac{27}{128}}\\[2pt]
-{\frac{27}{128}}&0&{\frac{27}{128}}\\[2pt]
{\frac{9}{128}}&0&-{\frac{9}{128}} 
\end{bmatrix}$
\\
\bottomrule[1.2pt]
\end{tabular}}
\end{table}

\end{appendix}

\bibliography{reference}
\bibliographystyle{siamplain}

\end{document}